\documentclass{article}
\usepackage{graphicx} 
\usepackage{amsthm,amsmath,amsfonts,amssymb}

\usepackage[a4paper, margin=1in]{geometry}

\usepackage{enumitem}

\usepackage{booktabs}

\usepackage{dsfont}
\usepackage{subcaption}   
\usepackage{xcolor}

\usepackage{appendix}

\usepackage{hyperref}
\hypersetup{
    colorlinks=true,
    linkcolor=blue,
    filecolor=magenta,
    urlcolor=cyan,
    pdftitle={Random permutations using GEPP},
    pdfpagemode=FullScreen,
    }

\usepackage{tikz}
\usetikzlibrary{positioning}
\usetikzlibrary{backgrounds}

\usepackage{ytableau}
\usepackage{multicol}

\usepackage{matlab-prettifier}

\usepackage{cleveref}

\def\V#1{{\mathbf #1}}

\def\GUE{\operatorname{GUE}}
\def\GOE{\operatorname{GOE}}
\def\Bern{\operatorname{Bern}}

\def\Unif{\operatorname{Unif}}
\def\Haar{\operatorname{Haar}}

\def\P{\mathbb{P}}

\def\B{\operatorname{B}}

\def\SO{\operatorname{SO}}

\def\ord{\operatorname{ord}}

\DeclareMathOperator*{\argmax}{arg\,max}

\newtheorem{lemma}{Lemma}
\newtheorem{proposition}{Proposition}
\newtheorem{theorem}{Theorem}
\newtheorem{corollary}{Corollary}
\newtheorem{conjecture}{Conjecture}
\newtheorem{remark}{Remark}

\allowdisplaybreaks

\title{Random permutations using GEPP}
\author{Kenji Gunawan\thanks{University of California, San Diego, \href{mailto:kegunawan@ucsd.edu}{kegunawan@ucsd.edu}} \and John Peca-Medlin\thanks{Department of Mathematics, University of California, San Diego, \href{mailto:jpecamedlin@ucsd.edu}{jpecamedlin@ucsd.edu}} \and Chenyang Zhong\thanks{Department of Industrial Engineering and Operations Research, University of California, Berkeley, \href{mailto:cyzhong@berkeley.edu}{cyzhong@berkeley.edu}}}
\date{}

\begin{document}

\maketitle

\begin{abstract}
    Gaussian elimination with partial pivoting (GEPP) remains the most widely used solver for dense linear systems $A \mathbf x = \mathbf b$ for $A \in \mathbb C^{n\times n}$. We study the permutation $\pi = \pi(A)$ that arises in the GEPP factorization $PA = LU$, encoded by the permutation matrix factor $P = P_\pi$. When the input matrix is random, so is $\pi$. For random scalar butterfly matrices of size $2^n$ (a recursively defined family originally introduced to eliminate the need for pivoting altogether), we give the exact GEPP factorization and fully classify the induced permutation as an element of a $2$-Sylow subgroup of $S_{2^n}$ contained in the separable permutations. Moreover, the uniform-angle model induces the uniform distribution on this subgroup. For the GOE, GUE, and iid Bernoulli models, the induced permutation is never exactly uniform for $n \ge 2$. We give the precise rate of departure from uniformity at the leading pivot for the GOE and GUE, and give evidence that this non-uniformity vanishes asymptotically in the permuton sense. In contrast, for banded random matrices of sublinear bandwidth, including the tridiagonal $\beta$-Hermite ensembles, the induced permutation converges to the diagonal permuton. We further show that the resulting pivot probabilities are sensitive to implementation choices: standard LAPACK routines compare complex pivot candidates using the $\ell^1$ rather than $\ell^2$ norm, changing the GUE(2) pivot probability from $1/\sqrt3$ to $2/3$. Together these results establish a new connection between random matrix theory and permutation combinatorics through numerical linear algebra.
\end{abstract}

\section{Introduction}
\label{sec:intro}

Gaussian elimination (GE) with partial pivoting (GEPP) is the most widely used solver for dense linear systems $A \V x = \V b$ for $A \in \mathbb R^{n \times n}$. For instance, GEPP is the default method for MATLAB's `\texttt{\textbackslash}' operator for generic dense systems and the standard solver in NumPy's linear algebra module. As such, it remains a staple in introductory linear algebra courses. GEPP iteratively transforms the system into upper triangular form via rank-1 updates, with row swaps applied if needed before each elimination step to ensure the pivot is maximal in magnitude in the leading column of the remaining untriangularized system. GEPP results in the matrix factorization $PA = LU$ in $\mathcal O(n^3)$ FLOPs, where $L$ is lower triangular, $U$ upper triangular, and $P = P_\pi$ a permutation matrix that encodes each of the row swaps used throughout. The induced permutation $\pi = \pi(A) \in S_n$ is iteratively formed in $n-1$ steps as a product of transpositions,
\begin{equation}\label{eq:GEPP perm}
    \pi(A) = (n-1 \ i_{n-1}) \cdots (2 \ i_2)(1 \ i_1), \qquad i_k = \min\left\{\argmax_{j \ge k}|A_{jk}^{(k)}|\right\},
\end{equation}
where $A^{(k)}$ is the intermediate matrix after $k-1$ elimination steps. If the active column is zero, we set \(i_k=k\), set the corresponding multipliers to zero, and skip the elimination update. Note that $i_k \ge k$ by construction. Although we formulate the pivot rule above for real matrices, it extends to complex (or quaternion) valued matrices once a norm is specified for evaluating entries as vectors in the corresponding real coordinate space; so different choices of norm can  produce different pivoting permutations.

In \cite{P24}, the second author studied the question of how many GEPP row swaps are needed for random matrices. This is equivalent to studying the number of cycles in the disjoint cycle decomposition of the permutation derived from the $P$ factor (or equivalently the number of orbits under the permutation action): a pivot movement is only needed if $i_k > k$, which corresponds to when $k$ is a non-maximal element in its disjoint cycle decomposition. In particular, we note $i_k = k$ then aligns precisely with when $k$ is maximal in its cycle (orbit) and can thus be used to count the number of cycles. 

That cycle-counting question is the starting point for the present paper, but it only sees the permutation $\pi(A)$ through a single numerical statistic. Here we take up the finer question of the permutation itself, asking which random matrix models yield permutations that are exactly or asymptotically uniform. For the random butterfly matrices studied in \cite{P24_gecp,P24,PT23,Tr19}, we ask which permutations $\pi(A)$ can arise and with what law. We additionally explore the impact of the choice of norm for pivot comparisons for non-real matrices.

\subsection{Main results.}

We outline the main contributions of the current text below.
\begin{itemize}
    \item \emph{Non-uniformity for classical ensembles, and its conjectured asymptotic vanishing.} For $A \in \mathbb C^{n \times n}$ drawn from the GOE, GUE, or an iid Bernoulli model, the induced GEPP permutation $\pi(A)$ is never exactly uniform for $n\ge 2$. We give the precise rate at which the leading pivot probability for $\GOE(n)$ and $\GUE(n)$ departs from $1/n$ (\Cref{p: goe pn,p: gue pn}), and, based on permuton-level numerical evidence, conjecture that this non-uniformity nonetheless vanishes as $n \to \infty$ for random Wigner and iid matrix models (Conjecture~\ref{conj:goe-permuton}).

    \item \emph{Exact factorization.} We give the closed-form GEPP factorization $PB = LU$ for general (nonsimple scalar) butterfly matrices $B$ (\Cref{p:ns factor}), extending the simple butterfly matrix case of \cite{PT23} to the full recursive family.

    \item \emph{Classification of the induced permutation group.} We show the permutations arising this way form a $2$-Sylow subgroup $B_n \subset S_{2^n}$ of separable permutations (i.e., avoiding the patterns $2413$ and $3142$), and that the uniform-angle model induces the uniform (Haar) measure on $B_n$ (\Cref{thm: unif ns}).

    \item \emph{Norm sensitivity in pivot comparisons.} We give a general formula for the $\GUE(n)$ no initial pivot probability under both the $\ell^1$ and $\ell^2$ norms used for complex pivot comparisons (\Cref{p: gue pn}), showing the two conventions produce genuinely different asymptotics. At $n=2$ this recovers, as a corollary, the discrepancy between the $\ell^2$-based value $1/\sqrt3$ and the value $2/3$ produced by LAPACK's $\ell^1$ convention (\Cref{sec:discrepancy}).

    \item \emph{Banded matrices and the diagonal permuton.} For the tridiagonal $\beta$-Hermite ensembles $H_\beta$, we compute the no initial pivot probability exactly and show it decays exponentially in $n$ for every $\beta>0$ (\Cref{p: beta pivot}). More generally, for any random matrix of sublinear bandwidth, the GEPP permutation converges to the diagonal permuton (\Cref{t: banded permuton}).
\end{itemize}

\subsection{Outline.} \Cref{sec:background} recalls the classical GEPP-uniformity results, introduces the permuton perspective, and gives the first quantitative departures from uniformity for $\GOE(n)$ together with a conjectured asymptotic uniformity. \Cref{sec:factorization} gives the exact GEPP factorization and the permutation-group classification of butterfly permutations, its diagonal-butterfly analogue (\Cref{sec: diagonal}), and a growth-factor conjecture (\Cref{sec:growth}). \Cref{sec:bias} revisits the first pivot computations for the tridiagonal $\beta$-Hermite ensembles (\Cref{p: beta pivot}) and shows that for banded matrices of sublinear bandwidth the GEPP permutation converges to the diagonal permuton (\Cref{t: banded permuton}), in contrast to the conjectured Lebesgue limit for the dense ensembles. It then isolates the role of the norm used for complex pivot comparisons, including a general formula for the $\GUE(n)$ pivot probability under each norm (\Cref{p: gue pn}), with the $\GUE(2)$/LAPACK discrepancy following as a corollary (\Cref{sec:discrepancy}). \Cref{sec:discussion} closes with open problems.

\section{Uniform permutations from GEPP}
\label{sec:background}

The study of random permutations started with uniform permutations. In 1938, Fisher and Yates in \cite{FY38} provided an algorithm to produce a uniform permutation by iteratively sampling a uniformly chosen remaining index to flip starting from a fixed initial permutation. Starting with the identity, then this aligns to produce a permutation in the form \eqref{eq:GEPP perm} to yield:
\begin{theorem}[\cite{FY38}]\label{t: fy}
    If $i_k \sim \Unif(\{k,k+1,\ldots,n\})$ are independent for $k = 1,\ldots,n$, then $(n \ i_n) \cdots$ $(2 \ i_2)(1 \ i_1) \sim \Unif(S_n)$.
\end{theorem}
A proof for this can be succinctly realized as an iterative application of the Subgroup algorithm of Diaconis and Shahshahani in \cite{DiSh87}:
\begin{theorem}[\cite{DiSh87}]\label{thm: subgroup alg}
Let $G$ be a compact group and $H$ a closed subgroup of $G$. Let
$x \sim \Unif(G/H)$ be uniform on the left cosets of $H$ in $G$, let
$y \sim \Unif(H\backslash G)$ be uniform on the right cosets of $H$ in
$G$, and let $z \sim \Unif(H)$ be independent of $x$ and $y$. Then
\(        xz,zy \sim \Unif(G).
   \)
\end{theorem}

\noindent Hence, if we view $S_{n-1} \subset S_n$ by identifying $S_{n-1}$ with the group of permutations of $[n]$ that fix 1, then for $i_1 \sim \Unif([n])$ and $\pi \sim \Unif(S_{n-1})$ we have $(1 \ i_1) \sim \Unif(S_{n-1} \backslash S_{n})$ as a coset representative, independently of $\pi$, and so $\pi (1 \ i_1) \sim \Unif(S_n)$. The left image in \Cref{fig:unif permuton} shows the diagram of pairs $(j,\pi(j))/n \subset [0,1]^2$ for $\pi_n \sim \Unif(S_n)$ for $n = 10^4$; as $n$ grows, this diagram will progressively fill in the entire unit box uniformly (e.g., see \cite{Borga21}). This can equivalently be stated as for $\pi_n \sim \Unif(S_n)$, then the measure
\begin{equation}\nonumber
    \label{eq: diag meas}
    \mu_{\pi_n} = \frac1n \sum_{j=1}^n \delta_{(j/n,\pi_n(j)/n)},
\end{equation}
which places equal weights at each of the points $(j,\pi_n(j))$ scaled to fit inside the unit square, limits weakly to the Lebesgue measure. Note if $\pi_n$ is random, then $\mu_{\pi_n}$ is a random measure. We note, however, $\mu_{\pi_n}$ is not a \textit{permuton}, which is a measure on the unit square with uniform marginals. Every permutation does yield a permuton: we can construct the corresponding permuton
\[
\hat \mu_{\pi_n} = n\sum_{j=1}^n \mathds 1_{A_j} \ dxdy, \qquad A_j = [(j-1)/n,j/n) \times [(\pi_n(j)-1)/n,\pi_n(j)/n).
\]
However, a straightforward check shows weak limiting (in probability) behavior of $\mu_{\pi_n}$ and $\hat \mu_{\pi_n}$ necessarily align, where we note a limit (if it exists) must itself be a permuton (since $\hat \mu_{\pi_n}$ is always a permuton and has uniform marginals). As seen in \Cref{fig:unif permuton}, a uniform permutation progressively fills in the unit square as the size grows. This contrasts with the right image in \Cref{fig:unif permuton}, which shows the corresponding fractal random diagram plot for a uniformly random \textit{separable} permutation. Separable permutations comprise a class of pattern-avoiding permutations, which avoid the patterns $2413$ and $3142$; equivalently, these are permutations built using only direct and skew sums ($\oplus/\ominus)$ of smaller separable permutations. Uniformly random separable permutations yield limiting Brownian separable permutons, which belong to a broader class of permutons with connections to Liouville quantum gravity and Schramm-Loewner evolutions \cite{bassino2018brownian,borga2023baxter,Borga2023,BorgaGwynneSun2025}.

\begin{figure}
    \centering
    \includegraphics[height=5cm]{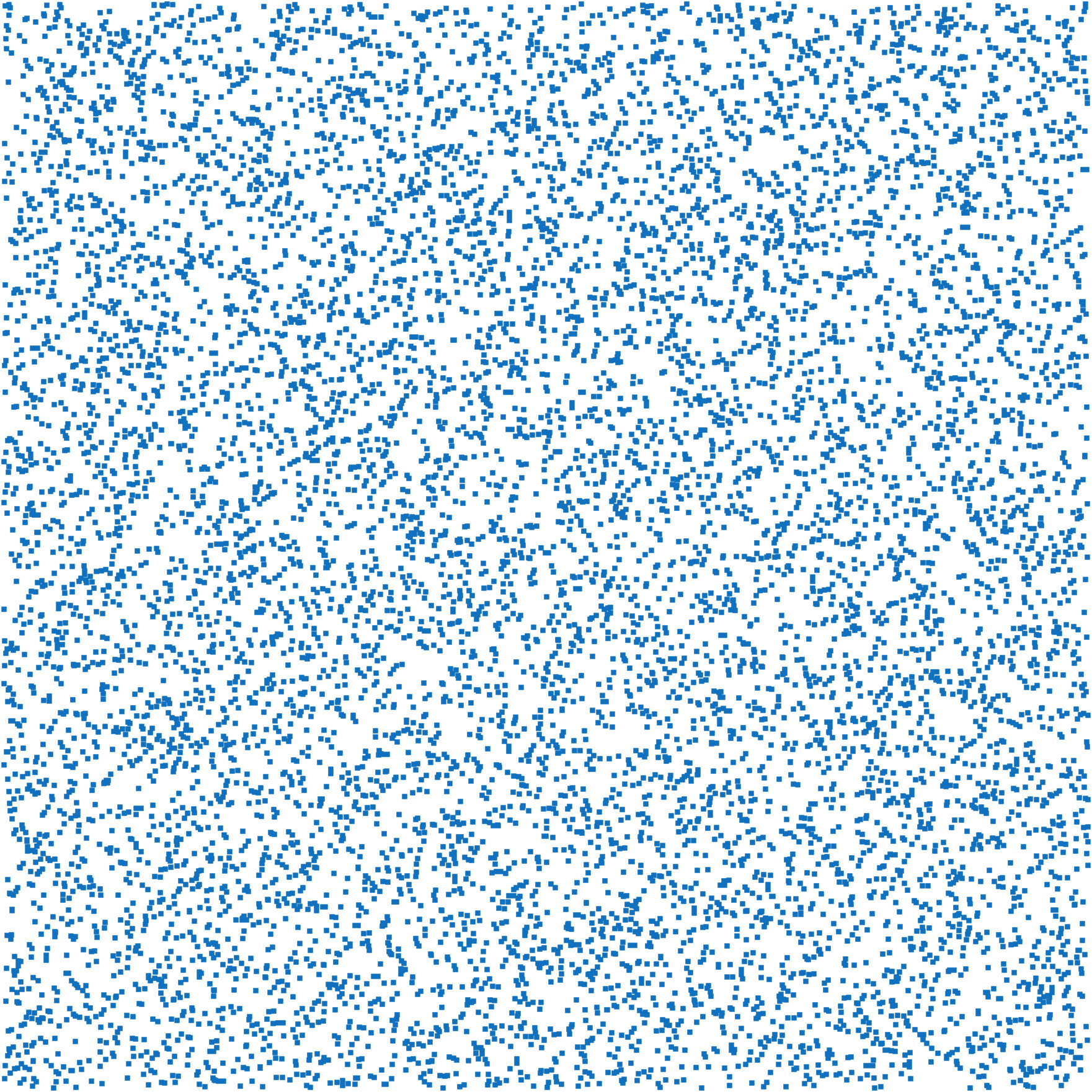}
    \quad
    \includegraphics[height=5cm]{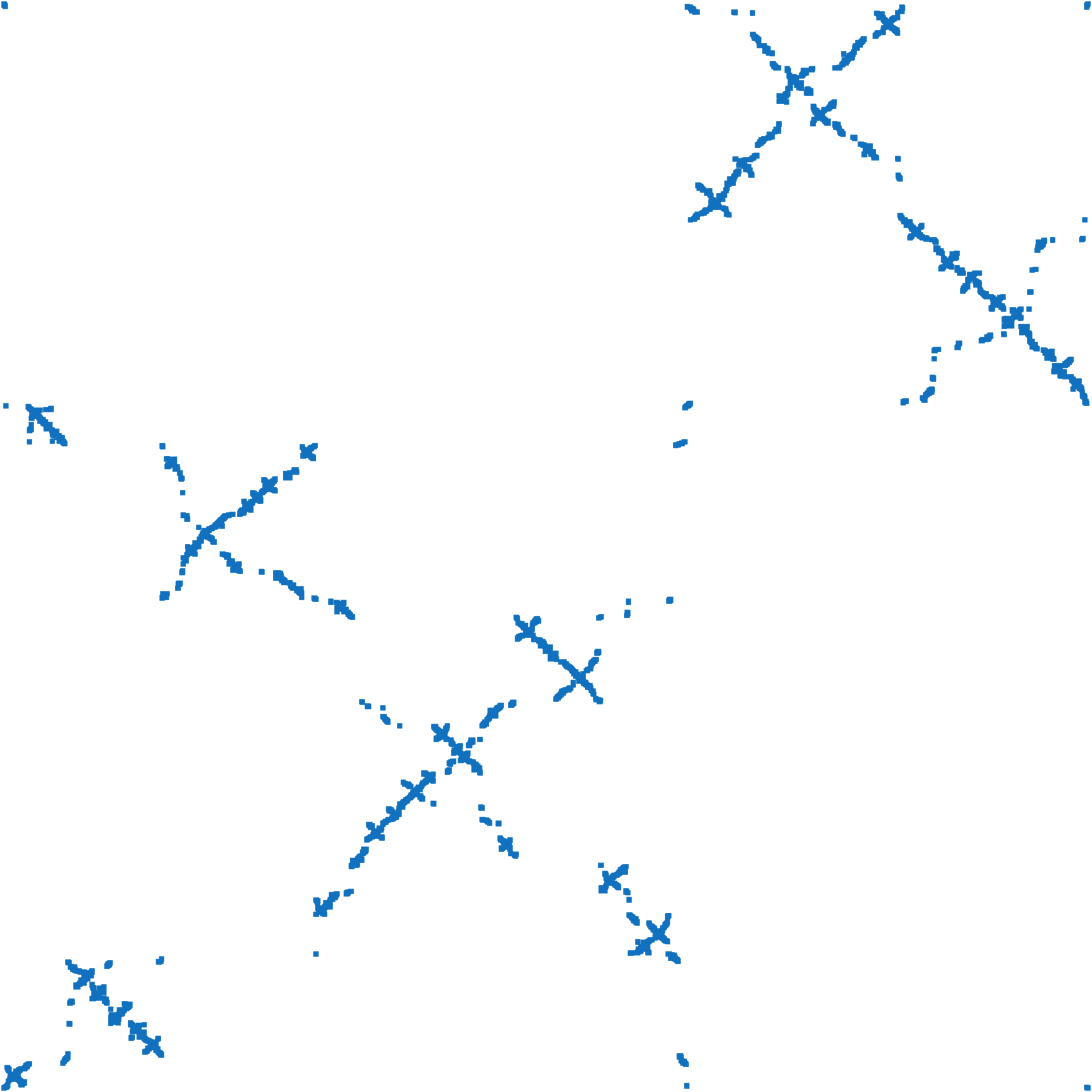}
    \caption{Diagram of pairs $(j,\pi(j))/n$ for $\pi \sim \Unif(S_n)$ (left) and a uniformly random separable permutation $\pi$ of order $n$ (right), for $n = 10^4.$}
    \label{fig:unif permuton}
\end{figure}

A natural question then is what type of random matrices produce uniformly random permutations using GEPP? In \cite{P24}, the second author established sufficient (but not necessary) conditions on $A$ such that $\pi(A) \sim \Unif(S_n)$.

\begin{theorem}[\cite{P24}]
\label{thm:unif_perm}
    If the first $n-1$ columns of $A \in \mathbb C^{n\times n}$ are independent and each column has independent and identically distributed (iid) entries from an absolutely continuous distribution, then $\pi(A) \sim \Unif(S_n)$.
\end{theorem}

\noindent Additional connections of GEPP and QR factorizations (along with the Subgroup algorithm), yield other random matrix transformations that also yield uniform GEPP permutations. For instance, if nonsingular \(A \in \mathbb R^{n \times n}\) has GEPP factorization \(PA = LU\) and QR factorization \(A = QR\) (for \(R\) upper triangular with positive diagonal entries), then
$$
PQ = L\widetilde U,\qquad \widetilde U = UR^{-1},
$$
is the GEPP factorization of \(Q\). Thus, \(\pi(Q)=\pi(A)\). Combining this observation with Stewart's method in \cite{stewart} of sampling Haar orthogonal (unitary) matrices as the orthogonal factor of the QR factorization of Ginibre matrices (with iid standard (complex) normal entries) (see also, e.g., \cite{Mezz}), together with \Cref{thm:unif_perm}, shows that Haar orthogonal (unitary) matrices have uniform GEPP permutation factors. The same invariance argument extends this to the following broader class:

\begin{corollary}[\cite{P24}]
\label{cor:unif_perm}
    Let $U,V \sim \Haar(G_n)$ be iid for $G_n$ one of the (special) orthogonal or unitary groups of order $n$. If $A = U $ or $A = UBV^*$ where $B$ is deterministic nonsingular (and real when $U,V$ are real), then $\pi(A) \sim \Unif(S_n)$.
\end{corollary}

\noindent Of course, this is not comprehensive for \textit{all} random matrix models that yield uniform GEPP permutation factors: if $\sigma \sim \Unif(S_n)$, then $\pi(P_\sigma^\top) = \sigma$.

For some other random matrices that do not meet the hypotheses of \Cref{thm:unif_perm} or \Cref{cor:unif_perm}, then the resulting associated GEPP permutation is not necessarily uniform. This can be seen for simple  examples with $A \in \mathbb C^{2 \times 2}$, where $\pi(A) = (1 \ i_1)$ and $2-i_1 \in \{0,1\}$ is necessarily a Bernoulli random variable $\Bern(q)$ for $q = \P(i_1 = 1)$, i.e., the probability that no pivot movement is needed. Using $A \sim \GOE(2)$ or $\Bern(p)^{2\times 2}$ (an iid matrix) (cf. \cite[Examples 1 and 2]{P24}), then $q$ takes the respective values $\frac2\pi \arctan({\sqrt 2}) \approx 0.60817$ and $p+(1-p)^2=1-p(1-p) \ge 3/4$. (\Cref{sec:discrepancy} outlines this also for $\GUE(2)$, which highlights a sensitivity to the norm evaluation of the entry $|A_{ij}^{(k)}|$ in \eqref{eq:GEPP perm} for complex values.) This differs from the uniform permutation behavior guaranteed, say, if $A_{ij} \sim N(0,1)$ iid, where then $q=1/2$.

A natural question then for these models is what can be determined for the associated GEPP permutation. For fixed $n\geq 2$ for each model, the induced random permutation is never truly uniform. This can be shown by only considering $i_1$ from \eqref{eq:GEPP perm}. Note first the following characterization of the case when $i_1 = 1$, that occurs when no initial pivot movement is needed:

\begin{lemma}\label{l: fixed}
    Let $\pi \in S_n$ have form given by \eqref{eq:GEPP perm}. Then $i_1 = 1$ if and only if $1$ is a fixed point of $\pi$, i.e., $\pi(1) = 1$.
\end{lemma}

\begin{proof}
    We have $i_1 = 1$ if and only if $1$ is maximal in its corresponding cycle factor in the cycle decomposition of $\pi$ if and only if 1 is the only number in its $\pi$-orbit if and only if $\pi(1) = 1$.  
\end{proof}

\begin{remark}
    This characterization of $i_1 = 1$ corresponding to 1 being a fixed point does not extend to any other case $i_k = k$: we only need to consider the $k$-cycle $\pi = (1 \ 2 \ \cdots k) = (k-1 \ k)\cdots (2 \ k)(1 \ k) \in S_n$ for $n \ge k$, as $i_k = k$ but $\pi(k) = 1$.
\end{remark}

By the Fisher-Yates product order, we have $i_1 = \pi^{-1}(1)$; so if $\pi$ is uniform then necessarily $i_1 \in [n] := \{1,2,\ldots,n\}$ must also be uniform. Thus a sufficient check to show a permutation is not uniform is to show $i_1$ is not uniform on $[n]$. Generalizing the $n=2$ case, we can ask about the evaluation for 
\[
p_n = \P(i_1 = 1) = \P(\pi(A)(1)=1),
\]
for when no pivot movement is needed in the first GEPP elimination step for random $A$, using \Cref{l: fixed}. If $\pi(A) \sim \Unif(S_n)$, then $p_n = 1/n$. So again a sufficient rejection of uniformity is showing $p_n \ne 1/n$. 

We first show an iid Bernoulli matrix can never yield a uniform permutation:

\begin{proposition}
    Let $A \sim \operatorname{Bern}(p)^{n \times n}$ for $p \in [0,1]$ fixed  and $n\ge 2$. Then $p_n = \Theta(1)$ and $\pi(A) \not \sim \Unif(S_n)$.
\end{proposition}

\begin{proof}
    Note first
    \[
    p_n = \P(i_1=1)
        =\P(A_{11}=1\ \mbox{or}\ A_{j1}=0\ \forall j)
        =p+(1-p)^n = \Theta(1).
    \]
    If $p = 0,1$, then $p_n = 1 > 1/n$ for all $n \ge 2$. Now assume $p \in (0,1)$. We note if $\pi(A) \sim \Unif(S_n)$ then $\P(i_1 = 1) = \P(i_1 = 2) = 1/n$ so that $\P(i_1 = 1) - \P(i_1 = 2) = 0$. We now compute also
    \begin{align*}
        \P(i_1=2)
        &=\P(A_{11}=0,A_{21}=1)
        =(1-p)p = p - p^2,
    \end{align*}
    so that
    $$
        p_n-\P(i_1=2)
        =p^2+(1-p)^n \ge p^2>0.
    $$
    This establishes $\pi(A)$ can never be uniform on $S_n$.
\end{proof}

This shows for an iid Bernoulli matrix $A$, then $p_n = \Theta(1)$. For $A\sim\GOE(n)$, we instead have $p_n=o(1)$ but $p_n=\omega(1/n)$:

\begin{proposition}
\label{p: goe pn}
Let $p_n = \P(i_1 = 1)$ for $A \sim \GOE(n)$. Then, for $n\ge2$,
$$
    p_n > \frac1n, \qquad  \P(i_1=j)=\frac{1-p_n}{n-1},\qquad j=2,\ldots,n.
$$
In particular,  $\pi(A) \not \sim \Unif(S_n)$ for any $n\ge2$, and
\begin{equation}\label{eq: asym pn}
p_n \sim \frac{\pi^{1/4}/\sqrt 2}{\sqrt n (\log n)^{1/4}}
\approx \frac{0.941}{\sqrt n (\log n)^{1/4}}.
\end{equation}
\end{proposition}

\begin{proof}
Writing $A=(G+G^\top)/\sqrt2$ for a real Ginibre matrix
$G$ (with iid $N(0,1)$ entries) recovers the $\GOE(n)$
distribution. Thus,
$A_{11}\sim N(0,2)$ is independent of
$A_{j1}\stackrel{\mathrm{iid}}{\sim}N(0,1)$ for $j\ge2$.
Writing $A_{11}^2=2X_1$ and $A_{j1}^2=X_j$ for $X_j\stackrel{\mathrm{iid}}{\sim}\chi_1^2$, then
\[
    p_n
    =\P\left(2X_1\ge M\right), \qquad M = \max_{2\le j\le n}X_j.
\]
If $2X_1$ were replaced by $X_1$, exchangeability would give probability $1/n$. Since $2X_1>X_1$ almost surely, and
$$
\P\left(2X_1 \ge  M >  X_1\right)>0,
$$
we have
$$
    p_n
    =\P\left(2X_1\ge M\right) = \P(X_1 \ge M) + \P(2X_1 \ge M > X_1) > \P(X_1 \ge M) = \frac1n.
$$
This establishes $\pi(A) \not\sim \Unif(S_n)$.

We now turn to the asymptotics for $p_n$. Conditioning on $X_1=x$ and using independence and $\P(\chi_1^2\le 2x)=\operatorname{erf}(\sqrt{x})$ gives
\[
    p_n
    =\frac1{\sqrt{2\pi}}\int_0^\infty
    x^{-1/2}e^{-x/2}
    \operatorname{erf}(\sqrt{x})^{n-1}\,dx.
\]
Substituting $s=\operatorname{erfc}(\sqrt{x})=1-\operatorname{erf}(\sqrt{x})$,
so that $ds = -e^{-x}/\sqrt{\pi x} \ dx$, converts this to
\[
    p_n
    =\frac1{\sqrt2}\int_0^1
    (1-s)^{n-1}e^{[\operatorname{erfc}^{-1}(s)]^2/2}\,ds.
\]
Since $\operatorname{erfc}(x)\sim e^{-x^2}/(x\sqrt\pi)$
as $x\to\infty$, inverting gives $[\operatorname{erfc}^{-1}(s)]^2
    =\log(1/s)-\frac12\log\log(1/s)
    -\frac12\log\pi+o(1)$ as $s\downarrow0$, and hence $e^{[\operatorname{erfc}^{-1}(s)]^2/2}
    \sim \pi^{-1/4}s^{-1/2}[\log(1/s)]^{-1/4}$. The factor $(1-s)^{n-1}$ localizes the integral near $s=1/n$,
where $[\log(1/s)]^{-1/4}\sim(\log n)^{-1/4}$ varies slowly
enough to pull out of the integral, leaving the Beta integral
\[
    \int_0^1(1-s)^{n-1}s^{-1/2}\,ds
    =B\left(\frac12,n\right)
    \sim\sqrt\pi\,n^{-1/2}.
\]
Combining constants yields \eqref{eq: asym pn}.
\end{proof}

\Cref{p: goe pn} shows scaling a single entry from the leading column by $\sqrt 2$ suffices to push the asymptotics for $p_n$ outside those of the uniform permutations.

A natural followup direction then could ask whether $\pi(A)$ can be asymptotically uniform in a global sense. This can then be addressed by determining whether the measure $\mu_{\pi_n}$ for GEPP induced random permutations $\pi_n$ weakly limits to the Lebesgue measure on the unit square, as seen in \Cref{fig:unif permuton}. \Cref{fig:goe bern permuton} shows scaled diagrams for a GEPP induced permutation $\pi(A)$ using $A \sim \GOE(n), \GUE(n),\Bern(1/2)^{n \times n}$ for $n = 10^4$ (with the GUE induced permutation using the $\ell^1$ norm evaluation for pivot comparisons; see \Cref{sec:discrepancy}), which suggests these limits converge to the Lebesgue measure on the unit square. Hence, although the local statistic for $p_n$ differs from that for uniform permutations for finite $n$, this image suggests this discrepancy may vanish in the global geometry as $n$ increases. $\GOE(n)$ and $\GUE(n)$ comprise standard Gaussian ensembles that fit the Wigner structure, which are random Hermitian matrices with iid entries with unit variance above the diagonal. $\Bern(p)^{n \times n}$ represents an iid ensemble outside the exact uniform case (unlike Ginibre).

This suggests the following broader conjecture: for broad classes of random matrix ensembles, the permutation induced by GEPP has a universal global geometry, even when its finite-\(n\) distribution is not uniform.

\begin{figure}
    \centering
    \includegraphics[height=5cm]{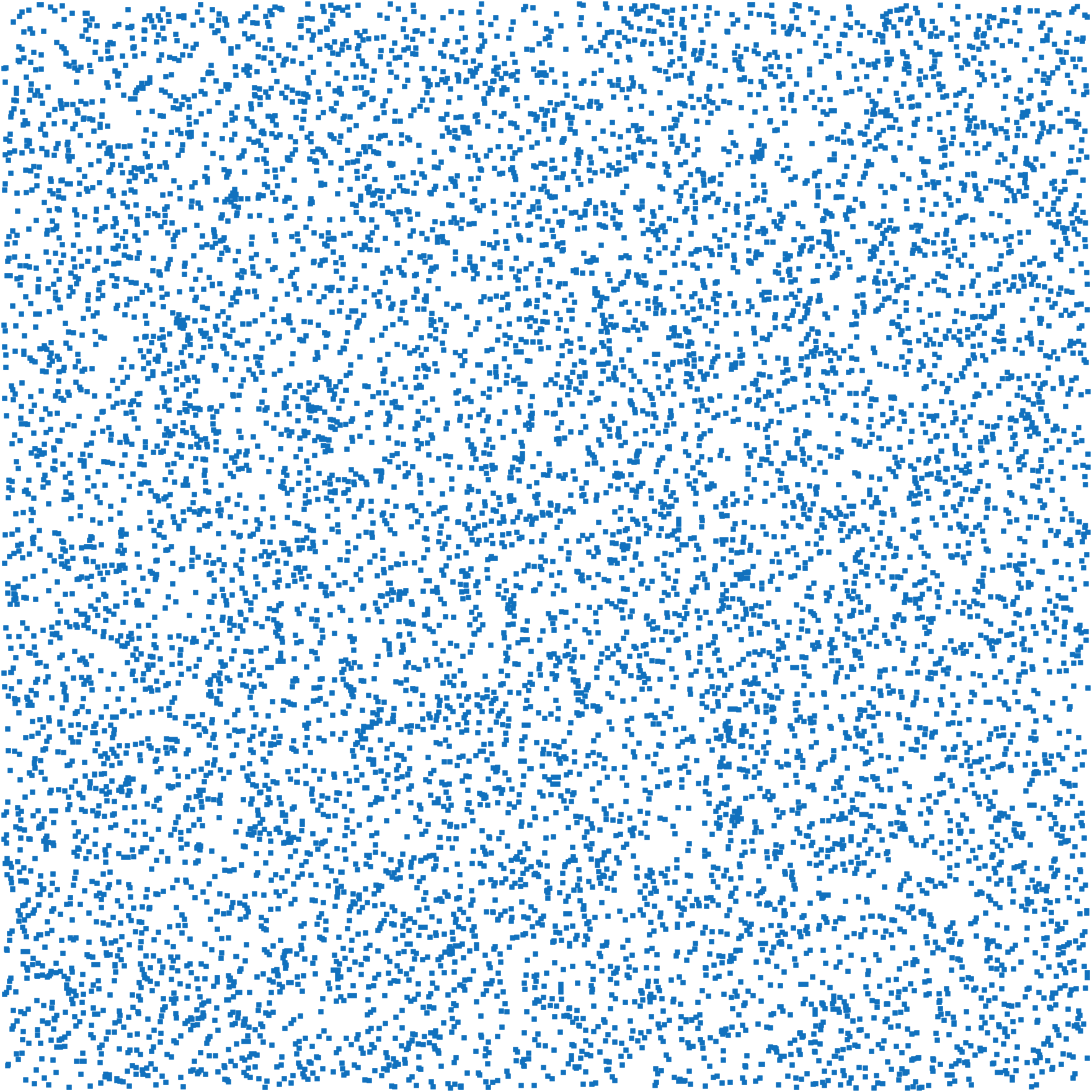} \quad
    \includegraphics[height=5cm]{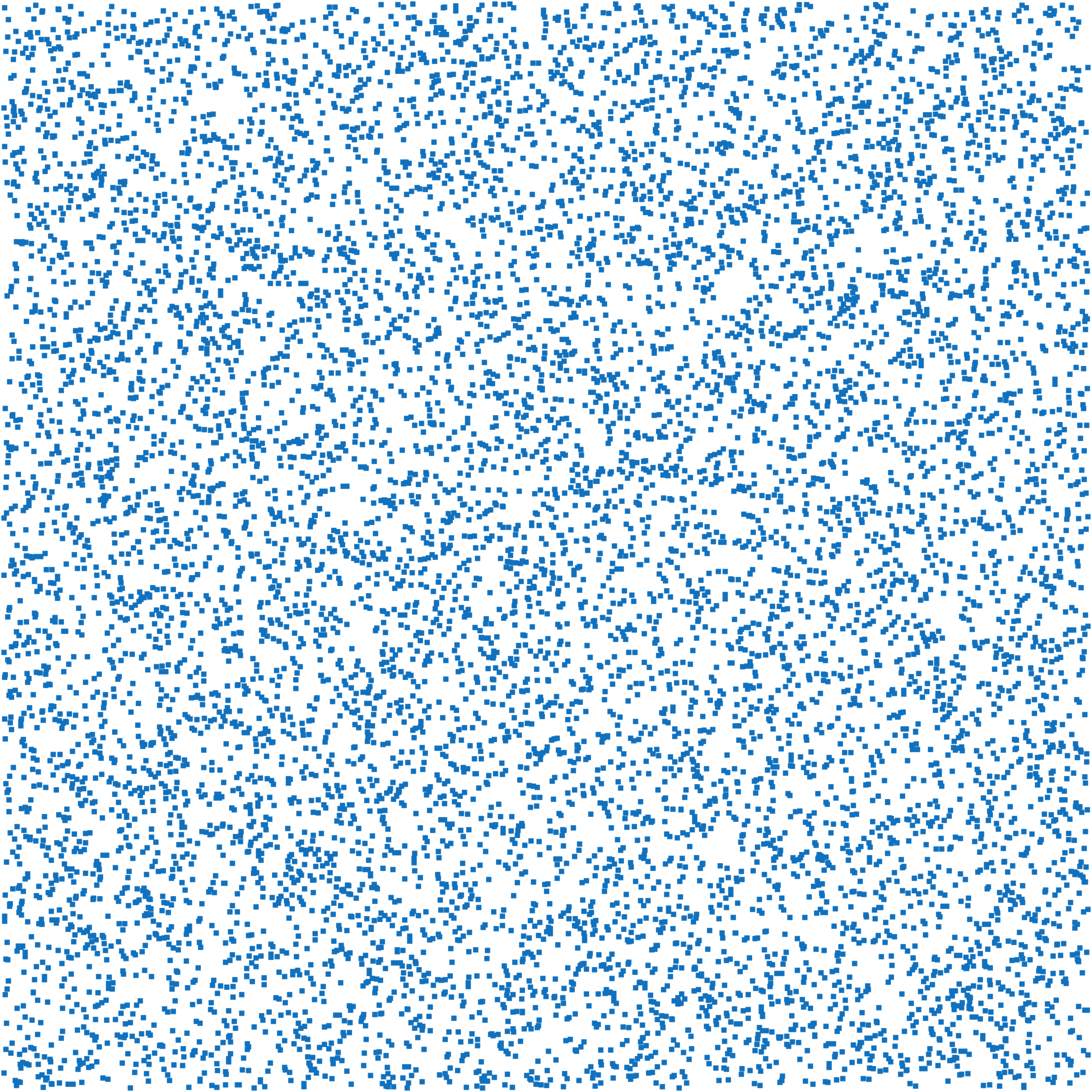} \quad
    \includegraphics[height=5cm]{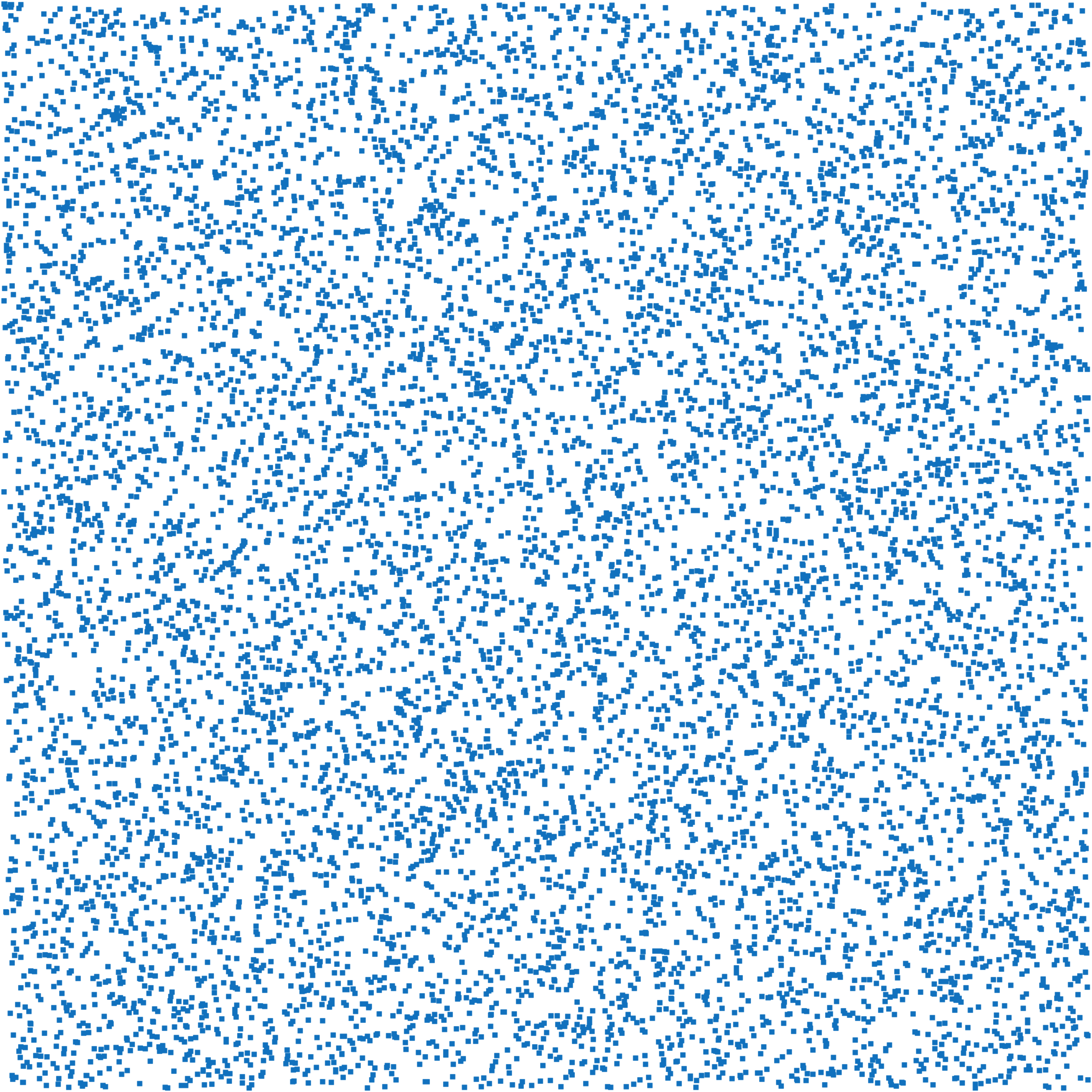}
    \caption{Diagram plots for GEPP induced permutations on $\GOE(n)$ (left), $\GUE(n)$ (middle), and $\Bern(1/2)^{n \times n}$ (right) for $n = 10^4$.}
    \label{fig:goe bern permuton}
\end{figure}

\begin{conjecture}
\label{conj:goe-permuton}
Let $A_n$ be a Wigner matrix or an iid matrix ensemble whose entry distribution is fixed independently of $n$ and is nondegenerate, with the pivot comparison and tie-breaking conventions described above. Let $\pi_n=\pi(A_n)$ be the permutation induced by GEPP. Then \(\mu_{\pi_n}\) converges weakly, in probability, to Lebesgue measure on \([0,1]^2\) as \(n\to\infty\).
\end{conjecture}

For the ensembles considered above, \(\pi_n\) is not uniform on \(S_n\) at finite \(n\geq 2\), yet the conjecture predicts the same global permuton limit as uniform permutations. Thus, distinct finite-\(n\) permutation laws may exhibit the same asymptotic global geometry.


\section{Butterfly permutations}
\label{sec:factorization}

Random \textit{butterfly permutations} are of the form $\pi(B)$ where $B$ is a random butterfly matrix. This section contains our main structural results: the exact GEPP factorization of general random butterfly matrices (\Cref{sec: simple butterfly,sec: ns butterfly}), the identification of the induced permutation group with a $2$-Sylow subgroup consisting of separable permutations, and a remark on why the corresponding growth-factor question, fully resolved in the simple case in \cite{PT23}, remains open here (\Cref{sec:growth}). We begin with a review of random butterfly matrices.

\subsection{Random butterfly matrices}
\label{sec:butterfly}

Random butterfly matrices were introduced by Parker to remove the \textit{need} of pivoting altogether when using GEPP \cite{Pa95}. Pivoting with GE can become a bottleneck that significantly slows down computation and often obstructs the use of parallelization and other acceleration strategies (cf. \cite{baboulin,Pa95}). Parker showed independent butterfly matrices $U,V$ of order $N = 2^n$ can be used to transform the nonsingular linear system $A \V x = \V b$ of order $N$ into the equivalent system $UAV^*\V y = U\V b$ and $\V y = V \V x$, which can  then be solved (with high probability in floating-point arithmetic, and almost surely in exact arithmetic) using GE with no pivoting (GENP); the recursive structure of (diagonal) butterfly matrices yields fast matrix-vector multiplication of order $O(N \log N)$ FLOPs, so that this preprocessing $UAV^*$ can be carried out without impacting the $O(N^3)$ leading order complexity of GE. An order $N=2^n$ butterfly matrix, $B$, which was studied extensively in \cite{P24_gecp,P24,PT23,Tr19}, is constructed recursively as $B = 1$ for $n=0$ and
\begin{equation}\label{def: butterfly}
    B = \begin{bmatrix}
        C & S \\ -S & C
    \end{bmatrix} \begin{bmatrix} A_1 & \V 0 \\ \V 0 & A_2 \end{bmatrix} = \begin{bmatrix}
        C A_1 & S A_2 \\ -S A_1 & C A_2
    \end{bmatrix},
\end{equation}
where $A_1,A_2$ are order $N/2$ butterfly matrices and $C,S$ are order $N/2$ real symmetric matrices that satisfy $CS = SC$ and the Pythagorean matrix equation $C^2 + S^2 = \V I$.  We will call butterfly matrices \textit{simple} when $A_1 = A_2$ holds at each recursive step. When $C,S$ are diagonal matrices, then their corresponding entries necessarily come in $(\cos\theta,\sin\theta)$ pairs for some $\theta$. By construction, then $B \in \SO(N)$. We will focus especially on when $C,S$  are also scalar matrices, so that $(C,S) = (\cos\theta,\sin\theta) \V I$ at each recursive step, and denote $\B(N)$ for the scalar butterfly matrices and $\B_s(N)$ for the simple scalar butterfly matrices. Under this lens, \eqref{def: butterfly} for scalar butterfly matrices $B \in \B(N)$ can be rewritten as
\begin{equation*}
    B = (R_\theta \otimes \V I_{N/2})(A_1 \oplus A_2), \qquad \mbox{where} \qquad R_ \theta = \begin{bmatrix}
\cos\theta & \sin\theta \\ -\sin\theta & \cos\theta
\end{bmatrix},
\end{equation*}
is the (clockwise) rotation matrix of angle $\theta$, using $\otimes$ for the Kronecker product of matrices and $\oplus$ for the direct sum. When $B \in \B_s(N)$ then $B$ further takes the form
\begin{equation*}
B = \bigotimes^{n}_{j=1} R_{\theta_j}.
\end{equation*}
In particular, we have for any $j = 0,\ldots,n$, then
\begin{align*}
    \B_s(N) = \bigotimes_{k=1}^n \SO(2) = \B_s(2^{n-j})\otimes \B_s(2^{j}), \quad \B(N) = (\B(2^{n-j}) \otimes \V I_{2^j}) \bigoplus_{k=1}^{2^{n-j}} \B(2^j).
\end{align*}
Let $\B^{(D)}_s(N)$ and $\B^{(D)}(N)$ denote the simple and general diagonal butterfly matrices, whose left-factor from \eqref{def: butterfly} has $(C,S) = \bigoplus_{j=1}^{N/2}(\cos\theta_j,\sin\theta_j)$ for each recursive step. Note then
\begin{equation}\label{def: diag factor}
    \begin{bmatrix}
        C & S\\ -S & C
    \end{bmatrix} = Q_n \left(\bigoplus_{j=1}^{N/2} R_{\theta_j} \right)Q_n^\top
\end{equation}
where $Q_n$ is the perfect shuffle matrix such that $Q_n(A \otimes B)Q_n^\top = B \otimes A$ for any $A \in \mathbb C^{N/2 \times N/2},B \in \mathbb C^{2 \times 2}$. For example, $Q_1 = \V I_2$, whereas $Q_2 = P_{(2 \ 3)} \in \mathcal P_4$, the set of permutation matrices of order $4$. It follows then the scalar butterfly matrices are formed using $n$ input angles in the simple case and $N-1$ input angles in the general butterfly case, while diagonal butterfly matrices are formed using $N-1$ and $nN/2$ respective input angles for the simple and general cases.

Random butterfly matrices $\B(N,\Sigma)$ are formed using the above recursion where each $(C,S)$ pair corresponds to a random matrix sampled from $\Sigma$. Let
\begin{align*}
    \Sigma_S &= \{(\cos\theta,\sin\theta)\V I_{2^k}: \theta \sim \Unif([0,2\pi)) \ \mbox{iid}\}\\
    \Sigma_D &= \{\bigoplus_{j=1}^{2^{k-1}}(\cos\theta_j,\sin\theta_j): \theta_j \sim \Unif([0,2\pi)) \ \mbox{iid}\}
\end{align*}
denote, respectively, when $C,S$ are formed using independent uniform angle scalar and diagonal matrices. In particular, then
\begin{equation*}
    \B_s(N,\Sigma_S) \sim \Haar\left(\bigotimes_{j=1}^n \SO(2)\right).
\end{equation*}

General $N = m^n$ order butterfly matrices can similarly be formed using the block structure
\begin{equation*}
    B = (A \otimes \V I_{N/m})\bigoplus_{j=1}^m A_j
\end{equation*}
for $A \in \mathbb C^{m \times m}$ and $A_j \in \mathbb C^{N/m \times N/m}$ for each $j$. Moreover, if $A,A_j$ are permutation matrices themselves, then so is $B$. See \cite{phd} for an overview of numerical and statistical properties of general butterfly matrices.

\subsection{Simple  butterfly permutations}\label{sec: simple butterfly}

When $B \sim \B_s(N,\Sigma_S)$, then $\pi(B)$ is a random simple (scalar) butterfly permutation. We will briefly outline then how one recovers the permutation factors for rotation matrices, and hence for simple scalar butterfly matrices. First recall the GENP factorization of a $2\times 2$ rotation matrix $R_\theta = L_\theta U_\theta$ where
\begin{equation*}
R_\theta = \begin{bmatrix}
\cos\theta & \sin \theta\\
-\sin\theta & \cos\theta
\end{bmatrix}, \qquad
L_\theta = \begin{bmatrix}
1 & 0\\
-\tan\theta & 1
\end{bmatrix}, \qquad
U_\theta = \begin{bmatrix}
\cos\theta & \sin \theta\\
0 & \sec\theta
\end{bmatrix}
\end{equation*}
for $\cos\theta\ne0$. If $\cos \theta = 0$, then $R_\theta$ does not have a GENP factorization. 

Furthermore, for $D = (-1) \oplus 1$ then $DR_\theta = R_{-\theta}D$, while $P_{(1 \ 2)} = R_{\pi/2}D$ so that $P_{(1 \ 2)} R_\theta = R_{\pi/2}DR_\theta = R_{\pi/2 - \theta}D$. Now note a GEPP pivot movement is needed on $R_\theta$ only when $|\cos\theta| < |\sin \theta|$ or $|\tan \theta| > 1$. (Throughout, we use the convention $|\tan\theta|=\infty$ when $\cos\theta=0$ and similarly $|\cot\theta| = \infty$ when $\sin \theta = 0$.) Define the pivot movement indicator
\begin{equation*}
e(\theta) =  \mathds 1(|\tan\theta| > 1) = \left\{\begin{array}{ll} 1 & \mbox{if $|\tan \theta| > 1$,} \\ 0 & \mbox{if $|\tan \theta| \le 1$,}
\end{array}\right.
\end{equation*}
and let $P_\theta = P_{(1 \ 2)}^{e(\theta)} = P_{e(\theta)\frac\pi2}$ and $D_\theta = D^{e(\theta)}$. It then follows
\begin{equation}\label{eq: 2x2 gepp}
P_\theta R_\theta = R_{\hat \theta}D_\theta = L_{\hat \theta} (U_{\hat \theta} D_\theta)
\end{equation}
is the GEPP factorization of $R_\theta$, with then $\hat \theta = \theta$ if $e(\theta) = 0$ (i.e., when no pivot movement is needed on $R_\theta$) and $\hat \theta = \frac\pi2 - \theta$ when $e(\theta) = 1$. This can then be used to establish the GEPP factorization of scalar butterfly matrices. Utilizing  properties of the Kronecker product (viz., its closure under transposition and inverses along with the mixed-product property), then the simple scalar butterfly permutations have the GEPP factorization $PB = LU$ where each factor is a Kronecker product of the associated Kronecker factors, with in particular $P = \bigotimes_j P_{\theta_j}$ when $B = \bigotimes_j R_{\theta_j}$. We can then define the simple scalar butterfly permutations  by
\begin{equation*}
     B_{s,n} = \bigotimes_{j=1}^n  S_2 =  B_{s,{n-j}} \otimes  B_{s,j} \subset S_N,
\end{equation*}
where $B_{s,1} = S_2$. (We use the $\otimes$ Kronecker product of permutations to denote the corresponding permutation matrix is realized as a Kronecker product of permutation matrices.) Since $S_2 = \langle {(1 \ 2)}\rangle \cong C_2$, where $C_k$ denotes the cyclic group of order $k$ (using the convention $\langle x \rangle = \{x^j: j = 0,1,2,\ldots\}$ to denote the subgroup of a group $G$ generated by $x \in G$), then $B_{s,n} \cong C_2^n$ by the mixed-product property. The bottom-left image in \Cref{fig:sparisity plot} shows the fractal diagram for a simple binary butterfly permutation.

As seen in \cite{P24}, since $\theta_j$ are iid uniform angles for $B \sim \B_s(N,\Sigma_S)$, then the generated GEPP permutation is uniform on $B_{s,n}$:
\begin{theorem}[\cite{P24}]\label{t: simple b}
    Let $B \sim \B_s(N,\Sigma_S)$. Then $\pi(B) \sim \Unif(B_{s,n})$.
\end{theorem}
The  Kronecker structure of $B_{s,n}$ can be further exploited to determine the explicit action of $\pi(B)$ for each $k \in [N]$. Let $\pi \in B_{s,n}$ and let $\pi_j \in   S_2$ such that $\pi = \bigotimes_{j=1}^n {\pi_{j}}$. We can write
\begin{equation*}
    \V e_k = \bigotimes_{j=1}^n \V e_{a_{j}+1},
\end{equation*}
where $a_{j} \in \{0,1\}$ are the coefficients of the binary expansion of $k-1 = (a_1a_2\cdots a_n)_2$, so
    \begin{equation*}
        k = 1 + \sum_{j=1}^{n} a_j2^{n-j}.
    \end{equation*}
For example, $\V e_5 = \V e_2 \otimes \V e_1 \otimes \V e_1$ where $5-1=4 = (100)_2$. By the mixed-product property,
    \begin{equation*}
        \V e_{\pi(k)} = P_\pi\V e_k = \bigotimes_{j = 1}^n P_{\pi_{j}} \V e_{a_{j}+1} = \bigotimes_{j = 1}^n \V e_{\pi_{j}(a_{j}+1)}
    \end{equation*}
so that
    \begin{equation}
    \label{eq: bperm rule}
        \pi(k) = 1 + \sum_{j=1}^{n}(\pi_{j}(a_j+1)-1)2^{n-j} = 1 + (b_1b_2\cdots b_n)_2
    \end{equation}
where $b_j = \pi_j(a_j+1)-1$. Hence, the action of $\pi$ can be fully realized as the action on the binary coefficients of the input. In particular, knowing the action of $\pi$ on a \textit{single} input $k$ completely determines all of the Kronecker factors, and hence $\pi$ itself. We therefore obtain the following.

\begin{proposition}
For any $i,j\in[N]$, there exists a unique $\pi \in B_{s,n}$ such that $\pi(i)=j$.
\end{proposition}

\subsection{Exact GEPP factorization and general butterfly permutations}\label{sec: ns butterfly}

The random (scalar) butterfly permutations are of the form $\pi(B)$ where $B \sim \B(N,\Sigma_S)$. We will now establish that these random permutations align necessarily with uniform permutations from the group of general butterfly permutations $B_{n} \subset S_{N}$, which is defined recursively as a semidirect product group
\begin{equation}\label{eq: Bn form}
 B_{n} = ( B_{n-1} \oplus  B_{n-1}) \rtimes_\varphi \langle (1 \ 2) \otimes 1_{N/2}\rangle
\end{equation}
where $\varphi(\pi_1 \oplus \pi_2) = \pi_2 \oplus \pi_1$ for $\pi_i \in B_{n-1}$, with $B_1 = \langle {(1 \ 2)}\rangle = S_2$ and $B_0 = \{1\}$. This can be rephrased as $ B_{n}$ is an iterated $n$-fold wreath product of $\langle (1 \ 2)\rangle \cong C_2$ (for $C_m$ denoting the cyclic group of order $m$), i.e., $$B_n \cong B_{n-1} \wr C_2 \cong C_2 \wr \cdots \wr C_2 = C_2^{\wr n}.$$ Since $|B_{{{n+1}}}| = 2|B_{n}|^2$, then $|B_{n}| = 2^{2^n - 1}$, and since $\ord_2 (2^n !) = 2^n - 1$ (using $d = \ord_r(s)$ for integers $r,s$ if $r^d \mid s$ and $r^{d+1} \nmid s$), then necessarily $B_{n}$ comprises a 2-Sylow subgroup of $S_{2^n}$ (which aligns with the $n$-fold wreath product presentation; see \cite{Kaloujnine_1948}), and also arises in the context of groups of $p$-adic automorphisms acting on rooted $p$-ary trees \cite{Abert_Virag_2005}. For example, we have for $n = 1$, then $B_1 = S_2 \cong C_2$, the cyclic group of order 2, and for $n = 2$, then $B_2 \cong D_8 \cong C_2 \wr C_2$, the dihedral group of order 8.

Moreover, we note that $B_n$ comprise \textit{separable} permutations. This follows directly from induction as the semi-direct product action from \eqref{eq: Bn form} has the trivial action correspond to the direct sum of the input lower order butterfly permutations and the nontrivial action the skew sum of the inputs. The top-left image in \Cref{fig:sparisity plot} shows the iid fractal structure of general butterfly permutations.

To determine $\pi(B)$ when $B \sim \B(N,\Sigma_S)$, we will determine the necessary form of the GEPP permutation matrix factor of $B \in \B(N)$. We first establish a more general result, which includes the general (and simple) butterfly matrices. This is an expansion of \cite[Proposition SM2.1]{PT23}, which provided a general GEPP form for Kronecker products of order $2^n$ involving a rotation matrix.

\begin{proposition}
    \label{p:ns factor}
Let $B = B(\theta,A_1,A_2) = (R_\theta \otimes \V I_m)(A_1 \oplus A_2) \in \mathbb R^{2m \times 2m}$, where nonsingular $A_1,A_2 \in \mathbb R^{m \times m}$ and $|\tan \theta| \ne 1$. Assume $A_1$ and $A_2$ have GEPP factorizations, given by $P_kA_k = L_kU_k$ for each $A_k$ where also $|(L_k)_{ij}| < 1$ for all $i > j$. Then $PB = LU$ is the GEPP factorization of $B$ where
\begin{align}
P &= (P_1 \oplus P_2)(P_\theta \otimes \V I_m) \label{eq: P ns butterfly}\\
L &= (P_1 \oplus P_2)(L_{\hat \theta} \otimes \V I_m)(P_1^\top \oplus P_2^\top)(L_1 \oplus L_2) \label{eq: L 1st}\\
&= \begin{bmatrix}
L_1 & \V 0 \\ -\tan \hat \theta P_2 P_1^\top L_1 & L_2
\end{bmatrix} \label{eq: L 2nd}\\
U &= (L_1^{-1} \oplus L_2^{-1})(P_1\oplus P_2)(U_{\hat \theta}D_\theta \otimes \V I_m)(A_1 \oplus A_2) \label{eq: U 1st} \\
&= \begin{bmatrix}
(-1)^{e(\theta)}\cos \hat \theta U_1 & \sin \hat \theta L_1^{-1}P_1 A_2 \\ \V 0 & \sec \hat \theta U_2
\end{bmatrix}. \label{eq: U 2nd}
\end{align}
\end{proposition}

\begin{figure}[t]
    \centering
    \includegraphics[height=5cm]{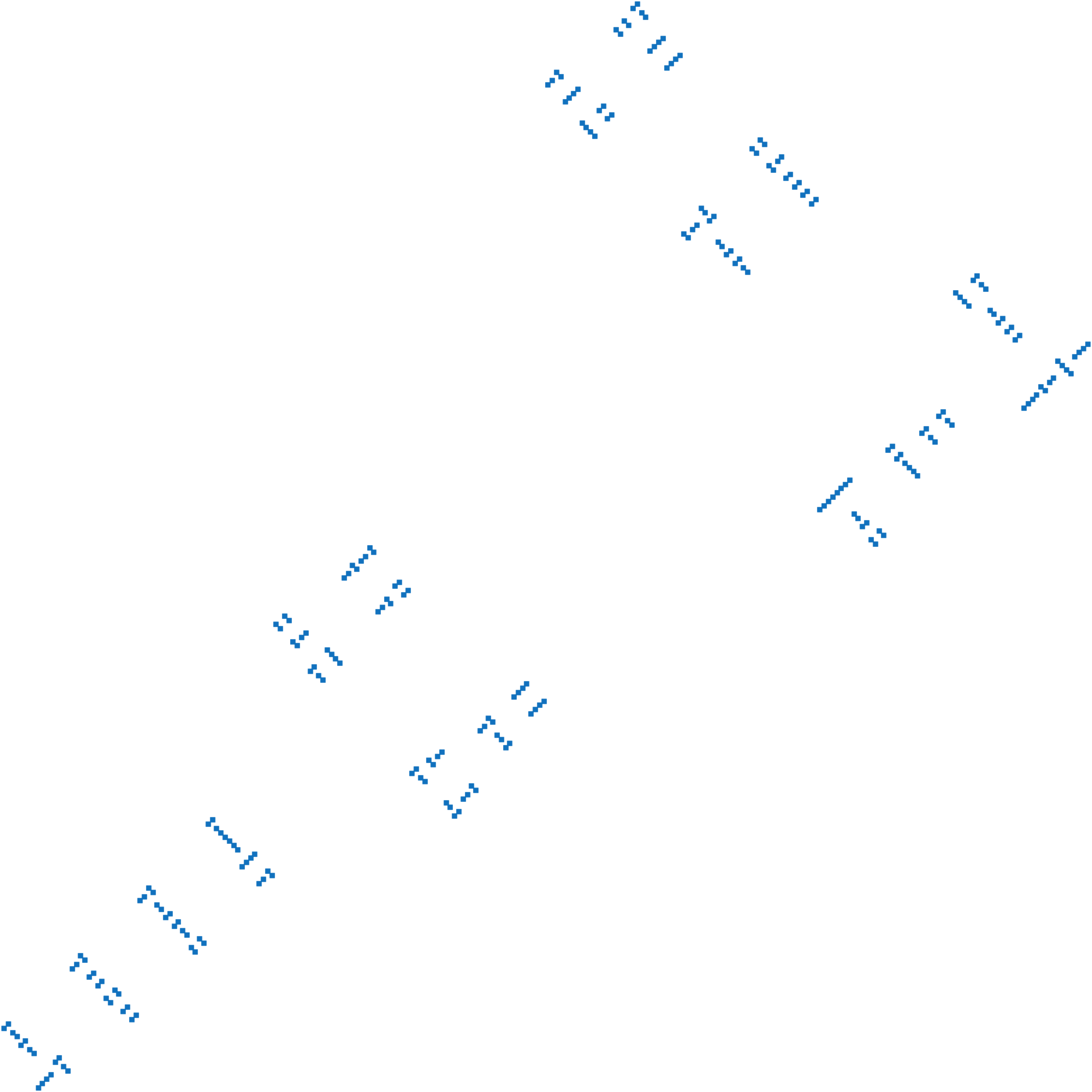} \qquad
    \includegraphics[height=5cm]{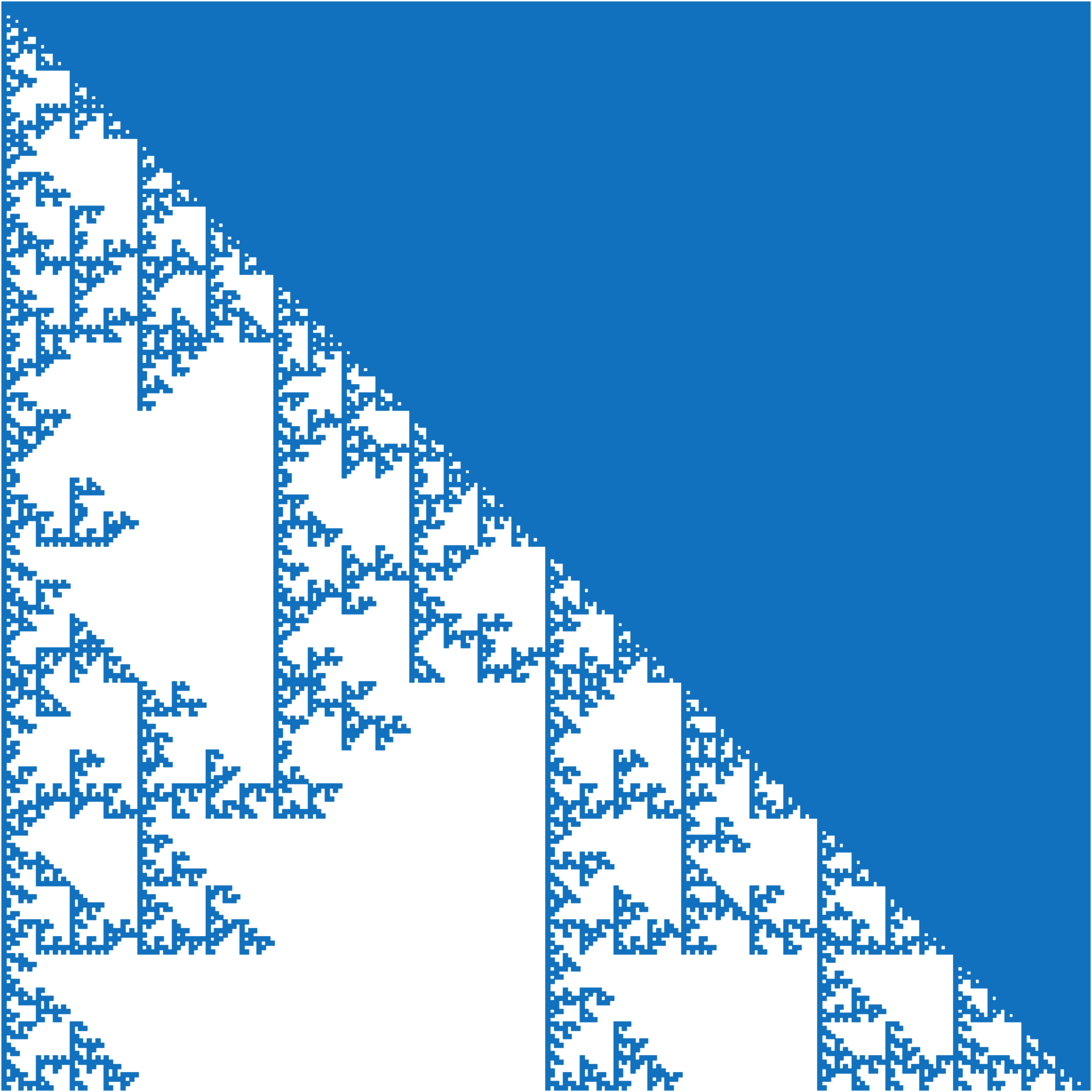} \vspace{1pc} \\
    \includegraphics[height=5cm]{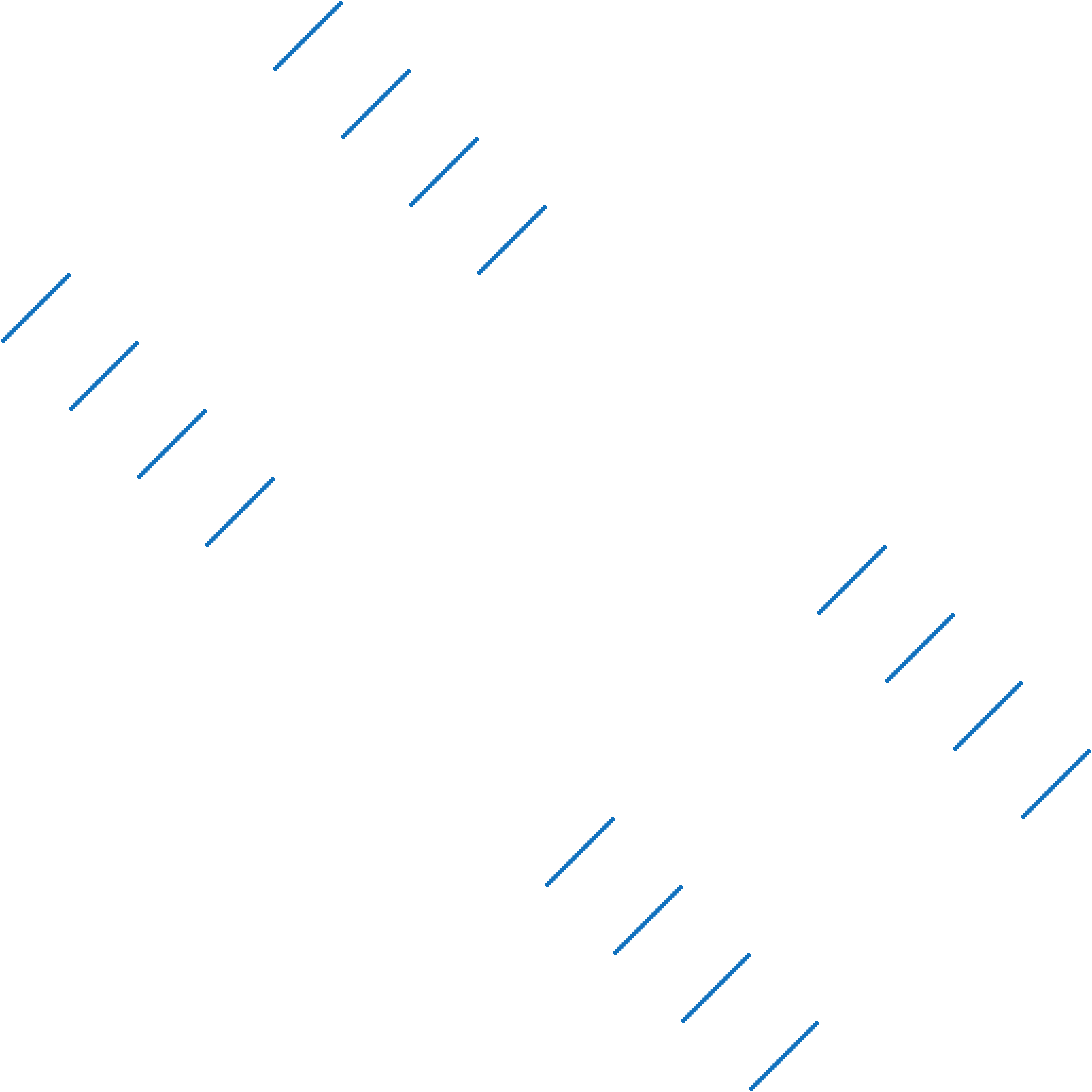}
    \qquad
    \includegraphics[height=5cm]{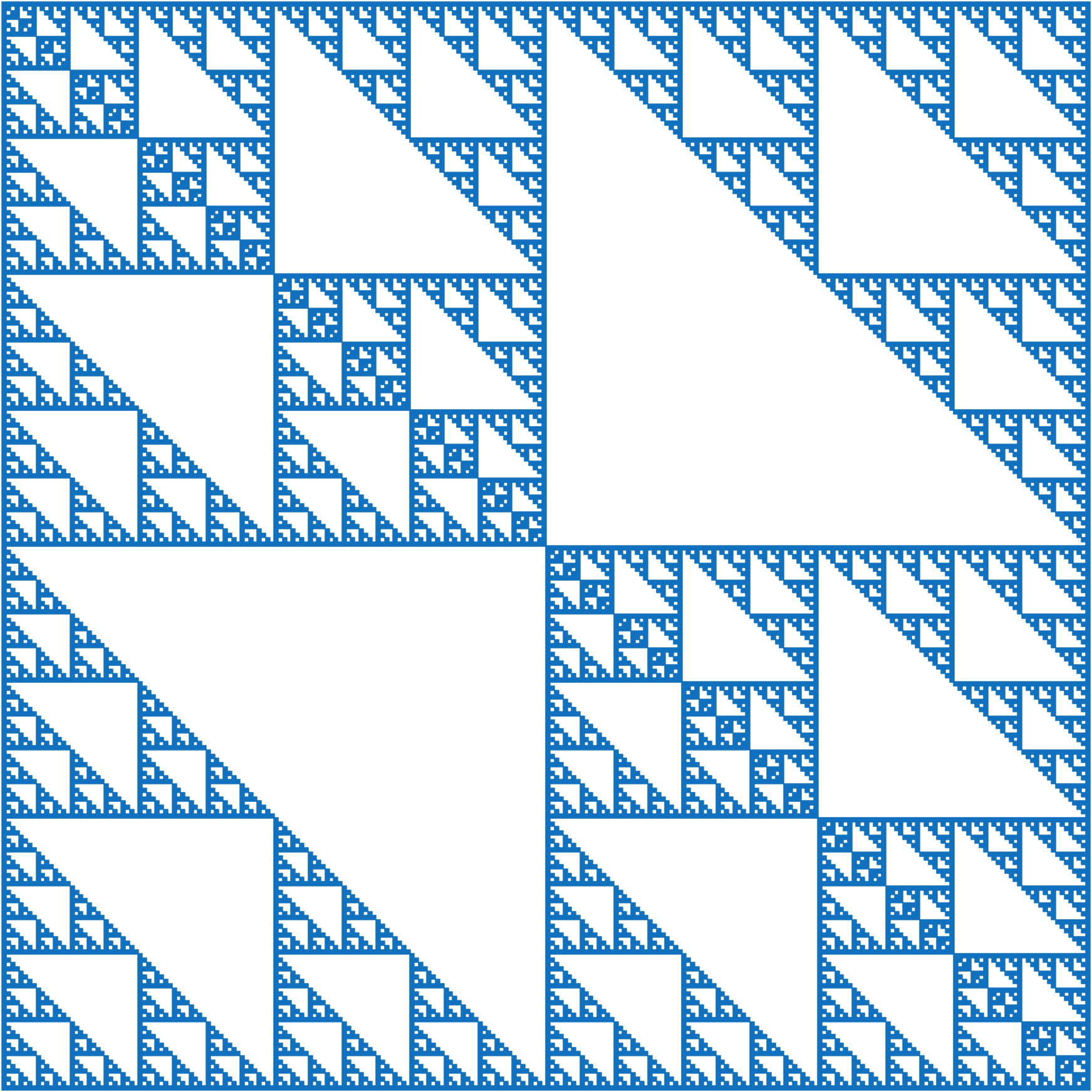}
    \caption{Sparsity plots of the GEPP induced permutation matrix $P$ and the matrix $L+U$ for general butterfly matrices (top) and simple butterfly matrices (bottom).}
    \label{fig:sparisity plot}
    \vspace{-1pc}
\end{figure}

\begin{proof}
The condition $|\tan \theta| \ne 1$ then eliminates the cases when there are pivot search ties  when using GEPP on $R_\theta$, and hence eliminates such ties between separate blocks on $B$. A straightforward computation verifies $PB = LU$ when using the forms given on \eqref{eq: L 1st} and \eqref{eq: U 1st}:
\begin{align*}
LU &= (P_1 \oplus P_2)(L_{\hat \theta} \otimes \V I_m)(U_{\hat \theta}D_\theta \otimes \V I_m)(A_1 \oplus A_2)\\
&= (P_1 \oplus P_2)(R_{\hat \theta}D_\theta \otimes \V I_m)(A_1 \oplus A_2)\\
&= (P_1 \oplus P_2)(P_\theta R_\theta \otimes  \V I_m)(A_1 \oplus A_2)\\
&=(P_1 \oplus P_2)(P_\theta \otimes \V I_m)(R_\theta \otimes \V I_m)(A_1 \oplus A_2)\\
&= PB,
\end{align*}
using the mixed-product property and \eqref{eq: 2x2 gepp}. Additional  straightforward checks  verify equality between forms \eqref{eq: L 1st} and \eqref{eq: L 2nd} as well as \eqref{eq: U 1st} and \eqref{eq: U 2nd}. To establish this is necessarily the (unique) GEPP factorization of $B$, we note this follows directly from  the fact $|L_{ij}| < 1$ for all $i > j$ (cf. \cite[Theorem SM1.4]{PT23} that characterizes this sufficient condition for real-valued matrices), which necessarily holds since $|(L_{k})_{ij}| < 1$ for $i > j$ for each $k$ and $|\tan \hat \theta| = \min(|\tan\theta|,|\cot\theta|) < 1$.
\end{proof}

See \Cref{fig:sparisity plot} for an image of the sparsity plots for $P$ and $L+U$ for the GEPP induced factors $PB = LU$ for general and simple butterfly matrices. (Note the sparsity plot of $P$ aligns with the diagram $(j,\pi(j))/2^n$ reflected across $y = 1/2$.) We note in particular the fractal structure of the corresponding butterfly permutations, with the simple case determined by direct or skew sums of identical block copies of lower order permutations. The deterministic sparsity pattern for the $L+U$ factor in the simple case, which manifests a reflected Sierpi\'nski triangle, aligns with the Kronecker product forms of each factor along with the initial sparsity patterns for $L_\theta$ and $U_\theta$. The random sparse $L$ factor and dense $U$ factor for general butterfly matrices follow from \Cref{p:ns factor} and induction, with the density following since also $L_1^{-1}P_1 A_2$ in the U factor in \eqref{eq: U 2nd} retains full density almost surely (when $A_1 \ne A_2$). Additionally, we note that each butterfly permutation has a random fractal structure that leads to a distinct limiting geometry outside the uniform universality class; we postpone a full classification of these limiting objects to future work. These rigidly self-similar fractal diagrams contrast also with those of uniform separable permutations (\Cref{fig:unif permuton}), which exhibit less rigidity in the absence of the additional group structure of butterfly permutations.

Recall for $B \sim \B(N,\Sigma_S)$, then $B  = (R_\theta \otimes \V I_{N/2})(A_1 \oplus A_2) = B(\boldsymbol \theta)$ for $\boldsymbol \theta = (\theta,\boldsymbol \theta^{(1)},\boldsymbol \theta^{(2)})$ where $A_i = B(\boldsymbol \theta^{(i)}) \sim \B(N/2,\Sigma_S)$ and $\theta_j^{(i)}\sim \Unif([0,2\pi))$ iid. Since $\P(|\tan \theta| = 1) = 0$, then $\pi(B) \in B_n$ almost surely for $B \sim \B(N,\Sigma_S)$ using GEPP by \Cref{p:ns factor}, where we further highlight the condition $|L_{ij}| < 1$ holds almost surely for $L$ the GEPP lower triangular factor. In particular, then $\operatorname{supp}\mathcal L(\pi(B)) = B_n$. Unlike simple scalar butterfly matrices, general butterfly scalar matrices are not closed under multiplication, and in particular, do not comprise a subgroup of $\SO(N)$. Nevertheless, these both induce subgroups of  $S_N$ through GEPP.

Using \Cref{p:ns factor} and straightforward induction, it follows that no GEPP pivot movements are needed if $|\tan\theta_k|<1$ for all $k$ (noting the explicitly removed pivot tie-condition if $|\tan \theta_k| = 1$). This further lines up with a similar result for simple butterfly matrices applied to the simple butterfly permutations, since $\pi_k=1$ (i.e., no pivot movement is needed) only if $|\tan\theta_k| <  1$ for $\bigotimes_{j=1}^n\pi_j\in  B_{s,n}$.

\begin{corollary}\label{cor: no pivot}
Let $B$ be a scalar butterfly matrix, with recursion angles $\theta_k$. If $|\tan\theta_k|<1$ for all $k$, then the GENP and GEPP factorizations of $B$ align, i.e., the GEPP factor has $P=\V I$.
\end{corollary}

\noindent The same conclusion holds with $|\tan\theta_k|\leq 1$ for all $k$ under the default GEPP tie-breaking convention. In particular, this applies almost surely when the angles are sampled independently and uniformly from $[0,2\pi)$, restricted to the event $|\tan\theta_k|\leq1$ for all $k$. A similar figure for \Cref{fig:sparisity plot} would yield $P=\V I$ and $L+U$ with full density on $U$ and $L$ with density matching the Sierpi\'nski triangle pattern from the simple butterfly case.

As seen in \cite{PT23}, the GEPP factorization for the simple case extends to also include alignment with the GE with rook pivoting (GERP) factorizations, which follows a pivot candidate search update path using successive column and then row searches until a candidate has maximal magnitude in both its row and column. This is not the case for general butterfly matrices, as can be seen by considering the first pivot search when the $A_2$ factor could have larger magnitude entries than $A_1$ that could result in a GERP column pivot movement. For example, for $B = (R_{\pi/4} \otimes \V I_2)(R_{\pi/4} \oplus R_{\pi/3}) \in \B(4)$, then
\begin{equation*}
    \|B_{:,1}\|_\infty = |B_{11}| = \frac12 = \sqrt{\frac28} < \|B_{1,:}\|_\infty = |B_{14}| = \sqrt{\frac38} = \max_{ij} |B_{ij}|.
\end{equation*}
Another future area of study can consider induced permutations using GE with other pivoting strategies, such as GERP or GE with complete pivoting (GECP).

We can now establish:

\begin{theorem}\label{thm: unif ns}
Let $B \sim \B(N,\Sigma_S)$. Then $\pi(B) \sim \Unif(B_n)$.
\end{theorem}

\begin{proof}
We will use induction on $n = \log_2 N$. For $n = 1$, the result holds from \Cref{t: simple b} since $\B(2,\Sigma_S) = \B_s(2,\Sigma_S) = \Haar(\SO(2))$ and $S_2 = B_1$. Assume the result holds for $n$. Let $B \sim \B(2^{n+1},\Sigma_S)$ and write $B = B(\theta,A_1,A_2) = (R_\theta \otimes \V I_{2^n})(A_1 \oplus A_2)$ for independent $\theta,A_1,A_2$ where $\theta \sim \Unif([0,2\pi))$ and $A_i \sim \B(2^n,\Sigma_S)$. By the recursive construction, all rotation angles are continuously distributed, so almost surely no angle is at a pole or tie value of the corresponding rotation matrix factorization. So all GEPP factorizations arising in the recursion are uniquely determined almost surely, independently of the tie strategy. Since $\P(|\tan \theta| \ne 1) = 1$, by \Cref{p:ns factor} we have
$$
    \pi(B) = (\pi_1 \oplus \pi_2)(\pi_\theta \otimes 1_{2^n}) \in B_{n+1}
$$
almost surely using GEPP with any tie strategy, where $\pi_i \in B_n$ and $\pi_\theta$ comprise the permutations from the corresponding GEPP permutation matrix factors, respectively, for $A_i$ and $R_\theta$. By the inductive hypothesis, $\pi_i \sim \Unif(B_n) = \Haar(B_n)$, with the corresponding GEPP factorizations well-defined almost surely, and so $\pi_1 \oplus \pi_2 \sim \Unif(B_n \oplus B_n)$. Moreover,
$$
    \pi_\theta\otimes 1
    \sim \Haar\bigl(\langle (1\ 2)\otimes 1_{2^n}\rangle\bigr)
    \sim \Unif\bigl((B_n\oplus B_n)\backslash B_{n+1}\bigr),
$$
that is, it is a uniform representative of the right cosets of the subgroup $B_n\oplus B_n$ in $B_{n+1}$ using \eqref{eq: Bn form}. This is independent of $\pi_1\oplus\pi_2$. It follows from the Subgroup algorithm (\Cref{thm: subgroup alg}) that
$$
    \pi(B)\sim\Haar(B_{n+1})=\Unif(B_{n+1}),
$$
as desired.
\end{proof}

We can use the explicit structure on ${\pi} \in  B_n$ from \eqref{eq: P ns butterfly} to determine a recursive formula to determine $\pi(k)$ for any $k \in [N]$. Let $\pi_i \in  B_{n-1}$ and $\pi_\theta \in B_1$ where ${\pi} = ({\pi_1} \oplus {\pi_2})(\pi_\theta \otimes 1)$. If $\pi_\theta = 1$, then $\pi = \pi_1 \oplus \pi_2$ while if $\pi_\theta = (1\ 2)$ then $\pi = \pi_2 \ominus \pi_1$. It follows
\begin{equation*}
    \pi(k) = \left\{
\begin{array}{ll}
\begin{array}{ll}
\pi_1(k) & \mbox{if $k \le N/2$}\\
\pi_2(k-N/2)+N/2 & \mbox{if $k > N/2$}\\
\end{array} & \mbox{if $\pi_\theta = 1$}\\ \vspace{-.5pc} \\
\begin{array}{ll}
\pi_2(k)+N/2 & \mbox{if $k \le N/2$}\\
\pi_1(k-N/2) \hspace{2.5pc} & \mbox{if $k > N/2$}
\end{array} & \mbox{if $\pi_\theta = {(1 \ 2)}$}
\end{array}    \right.
\end{equation*}
This characterization of \(\pi(k)\) explains the \((00,11)\) and \((01,10)\) structure visible in each occupied subsquare in \Cref{fig:sparisity plot}, which applies recursively to further subdivisions of each occupied square; the simple permutation action of \eqref{eq: bperm rule} manifests as the identical \((00,11)\) and \((01,10)\) subdivisions across each smaller occupied subsquare. Additional properties for $B_n$ are explored in terms of  longest increasing subsequences and number of cycles in \cite{pz_lis} and heights of binary search trees generated by these permutations in \cite{PZ2025heights}.

\subsubsection{Growth factors of random general butterfly matrices}
\label{sec:growth}

The \textit{growth factor}, $\rho(A) = \max_{k}\|A^{(k)}\|_{\max}/\|A\|_{\max}$, quantifies the maximal entry growth encountered during GE and controls numerical stability under finite precision (see \cite{urschel_notices} and references therein). In the simple case, \cite{PT23} fully characterizes the GEPP growth factor of $B \sim \B_s(N,\Sigma_S)$, showing the maximal entrywise growth is always attained in the final factor $U$ rather than in any intermediate matrix $B^{(k)}$, i.e., $|U_{NN}| = \|U\|_{\max} = \max_k \|B^{(k)}\|_{\max}$. This enables a closed form representation of the growth factor itself. The proof exploits the fact that the Kronecker structure of $\B_s(N)$ gives closed forms for \emph{every} intermediate GEPP factor $B^{(k)}$, not just the final $L,U$ pair for $PB = LU$ the GEPP factorization.

That argument does not extend to the general butterfly case. Fix $B = B(\theta,A_1,A_2)$ as in \Cref{p:ns factor} and consider eliminating the leading $m = N/2$ columns. Since these columns are $\begin{bmatrix}\cos\theta A_1\\-\sin\theta A_1\end{bmatrix}$, the first $m$ pivot decisions depend only on $A_1$ and the root angle $\theta$, and follow the same recursion as in the simple-matrix case. However, the corresponding row operations also act on the trailing $m$ columns, which are contributed by $A_2$ (specifically $\sin\theta A_2$ and $\cos\theta A_2$). Since $A_1 \ne A_2$ in general, there is not a clear reason for these elimination steps, built from the GEPP factors of $A_1$, to reduce the right block involving $A_2$ to any other object with a recursive closed form; the two halves of the recursion stop communicating in a way that induction alone can track. 

Yet, we still expect the growth to be controlled by the final $U$ factor alone. Each experiment we have run is consistent with the growth being attained somewhere by $U$; for instance, using double precision, every one of $10^3$ uniformly sampled butterfly matrices of order 512 attained maximal growth at this final $U$ entry. This further suggests:

\begin{conjecture}\label{conj:growth}
    For $B \sim \B(N,\Sigma_S)$, the entrywise GEPP growth factor is almost surely attained somewhere in the final factor $U$, i.e., $\rho(B) = \|U\|_{\max}/\|B\|_{\max}$.
\end{conjecture}

\subsection{Diagonal butterfly permutations}\label{sec: diagonal}

Unlike in the scalar butterfly matrix case, the GEPP factorization does not go through as directly for diagonal butterfly matrices. When no ties are encountered (i.e., when $|L_{ij}| < 1$ when $i > j$ for real-valued nonsingular matrices), then GEPP yields unique $L$ and $U$ factors for any row pivot permutation multiples of $A = LU$ (i.e., if $Q$ is a permutation matrix, then $B = QA$ has GEPP factorization $Q^\top B = LU$). However, each GEPP matrix factor is not necessarily invariant if the initial system uses \textit{column} permutations. (For example, GEPP pivoting is agnostic of any columns after the first $n-1$ of $A \in \mathbb C^{n \times m}$ for $m \ge n$, so permuting these columns has no impact on the GEPP permutation matrix factor.) The diagonal butterfly factors that involve conjugation by perfect shuffle permutation matrices  (permutation $q_k$ that corresponds to the permutation matrix $Q_k$ such that $Q_k(A \otimes B)Q_k^\top = B \otimes A$; cf. \eqref{def: diag factor}) thus obstruct a straightforward GEPP factorization result like \Cref{p:ns factor}.

All of this said, we can still define the diagonal butterfly permutations to again be the permutations corresponding to the GEPP permutation factors of random diagonal butterfly matrices, which we will denote by $B_{s,n}^{(D)}$ and $B_{n}^{(D)}$, respectively, for the simple and general diagonal butterfly permutations. Since $\B_s(N) \subset \B_s^{(D)}(N) \subset \B^{(D)}(N)$ and $\B_s(N) \subset \B(N) \subset \B^{(D)}(N)$, then $B_{s,n} \subset B_{s,n}^{(D)} \subset B_{n}^{(D)} \subset S_N$ and $B_{s,n} \subset B_n \subset B_n^{(D)} \subset S_N$. Note one can see $B_{s,n} \ne B_{s,n}^{(D)}$ and $B_n \ne B_n^{(D)}$ by considering first $N = 4$ (since $N = 2$ has equality everywhere since $B_{s,1} = S_2$). We note $B \in \B_s^{(D)}(4)$ of the form
\begin{align*}
    B = Q_2(R_{\pi/3}\oplus \V I_2)Q_2(\V I_2 \otimes R_{\pi/4})
    = \frac1{\sqrt 8} \begin{bmatrix}
        1 & 1 & \sqrt 3 & \sqrt 3\\
        -2 & 2 & 0 & 0\\
        -\sqrt 3 & -\sqrt 3 & 1 & 1\\
        0&0&-2 & 2
    \end{bmatrix}
\end{align*}
 has GEPP factorization $PB = LU$ for $P = P_{(1 \ 3 \ 2)} = P_{(2 \ 3)(1 \ 2)}$ and
\begin{equation*}
    L = \begin{bmatrix}
        1 & 0& 0& 0 \\ \frac{\sqrt{3}}2 & 1 & 0& 0\\
        -\frac12 & -\frac1{\sqrt 3} & 1 & 0\\ 0 & 0 & -\frac{\sqrt 3}2 & 1
    \end{bmatrix} \quad \mbox{and} \quad U = \frac1{\sqrt 8}\begin{bmatrix}
        -2 & 2 & 0 & 0\\0 & -2\sqrt 3 & 1 & 1\\ 0 & 0 & \frac{4}{\sqrt 3} & \frac{4}{\sqrt 3}\\0&0&0&4
    \end{bmatrix}.
\end{equation*}
Hence, $(1 \ 3 \ 2) \in B_{s,2}^{(D)}$. By considering $\pi_k = \pi(\V I_{2^k} \otimes B) \in B^{(D)}_{s,{k+2}}$ such that $P_{\pi_k} = \V I_{2^k} \otimes P_{(1 \ 3 \ 2)}$ so $\pi_k^3 = 1$,
then it follows $B^{(D)}_{s,n} \subset B^{(D)}_n$ always contains an element of order $3$ for all $n \ge 2$. Since all elements of $B_{s,n} \subset B_n$ must have order a power of 2 (since they belong to a 2-Sylow subgroup of $S_{2^n}$), then the inclusions $B_{s,n} \subset B^{(D)}_{s,n}$ and $B_n \subset B^{(D)}_n$ are necessarily strict.

Recall if one applies GEPP to a diagonal multiple of a permutation matrix $DP_\pi$ for $D$ a nonsingular diagonal matrix and $\pi \in S_n$, then one again necessarily recovers the inverse of the associated permutation matrix as the permutation factor, i.e., $\pi(DP_\pi) = \pi^{-1}$. If $\boldsymbol \theta$ is a vector with $\theta_j \in \{0,\frac\pi2\}$ for each $j$, then $B(\boldsymbol \theta)$ is a diagonal multiple of a permutation matrix (e.g., $R_{\pi/2} = P_{(1 \ 2)}((-1) \oplus 1)$), and so $\tilde B = \{\pi(B(\boldsymbol{\theta})): \theta_j \in \{0,\frac\pi2\}\}$ is a subset of the butterfly permutations. Since $|\tilde B| = 2^n$ when $B(\boldsymbol{\theta}) \in \B_s(N)$ and $|\tilde B| = 2^{2^n - 1}$ when $B(\boldsymbol{\theta}) \in \B(N)$, then necessarily $\tilde B$ aligns with $B_{s,n}$ and $B_n$; so no more butterfly permutations are possible than those generated using input angles $\{0,\frac\pi2\}$ for the scalar butterfly matrices. We reserve the analogous set of questions for diagonal butterfly permutations to future work.

\section{Pivot comparisons: norm selection and banded models}
\label{sec:bias}

In this section, we revisit the GEPP first pivot computations, starting with the tridiagonal $\beta$-Hermite ensembles and the banded matrices they motivate, and ending with the sensitivity to the norm used to compare complex entries in pivot searches.

The GOE, GUE, and GSE matrices are also called $\beta$-Hermite ensembles, where $\beta = 1,2,4$, respectively. These particular $\beta$-ensembles arise naturally in physics and random matrix theory applications (e.g., see \cite{Mehta}). This $\beta$ parameter explicitly manifests in terms of the eigenvalue problem for each ensemble, where the $\beta$-Hermite ensembles each have joint density 
\begin{equation}
    f_\beta(\boldsymbol \lambda) = c_\beta \prod_{i < j} | \lambda_i - \lambda_j|^\beta \exp\left(- \sum_{j = 1}^n \lambda_j^2/2\right), \qquad  c_\beta = \frac1{(2\pi)^{n/2}} \prod_{j=1}^n \frac{\Gamma\left(1 + \frac\beta 2\right)}{\Gamma\left(1 + \frac\beta 2 j\right)}.
\end{equation}
While the standard $\beta = 1,2,4$ models then have corresponding explicit matrix constructions (see also \cite{Mehta}) that match these eigenvalue distributions, general $\beta$-ensembles did not have a corresponding matrix model until Dumitriu and Edelman provided such a construction of tridiagonal real-valued random ensembles in \cite{DE02} for each $\beta > 0$. Their construction used Householder tridiagonalization of the standard Gaussian ensembles to motivate the tridiagonal form
\begin{equation} H_\beta \sim \frac1{\sqrt 2} \begin{bmatrix} N(0,2) & \chi_{\beta(n-1)}\\ \chi_{\beta(n-1)} & N(0,2) & \chi_{\beta(n-2)}\\ & \ddots & \ddots & \ddots\\ && \chi_{2\beta} & N(0,2) & \chi_{\beta}\\ &&& \chi_\beta & N(0,2) \end{bmatrix}. \end{equation}
The diagonal and distinct subdiagonal entries are mutually independent. With our conventions, $H_1$ has the same eigenvalue law as GOE$/\sqrt 2$, while $H_2$ has the same eigenvalue law as the GUE normalization used here. These models are obtained through tridiagonalization of the corresponding dense ensembles. Here, we further emphasize that these random matrix models are \textit{real}-valued, unlike the GUE and GSE constructions.

\begin{figure}[t]
    \centering
    \includegraphics[width=0.7\linewidth]{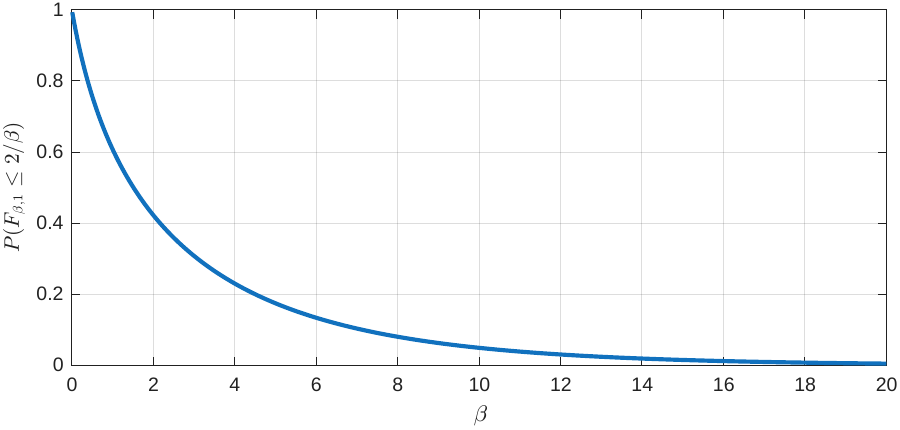}
    \caption{Plot of no pivot probabilities $p_2({\beta})$ for $H_\beta$}
    \label{fig:beta pivots}
\end{figure}

\begin{proposition}\label{p: beta pivot}
    Let $A \sim H_\beta$ be the $n \times n$ tridiagonal $\beta$-Hermite ensemble and $p_n(\beta) = \P(i_1 = 1)$. Then for every $\beta > 0$ and $n \ge 2$,
    \[
    p_n(\beta) = I_{2/3}\left(\frac{\beta(n-1)}2,\frac12\right) \sim \sqrt{\frac6{\pi\beta(n-1)}}\left(\frac23\right)^{\beta(n-1)/2},
    \]
    where $I_z(a,b)$ is the regularized incomplete beta function. In particular, $p_n(\beta)$ decays exponentially in $n$, faster than the order $1/n$ decay for a uniform permutation.
\end{proposition}

\begin{proof}
    The first pivot is $i_1 = 1$ exactly when $|A_{11}| \ge |A_{21}|$. Since $A_{11}^2 \sim \chi_1^2$ and $A_{21}^2 \sim \chi^2_{\beta(n-1)}/2$ are independent, this can be evaluated exactly using the $F$ distribution, where $F_{\mu,\nu} \sim (\chi^2_\mu/\mu)/(\chi^2_\nu/\nu)$ for independent numerator and denominator, with cumulative distribution function
    \[
    \P(F_{\mu,\nu} \le x) = I_{\mu x/(\mu x+\nu)}\left(\frac\mu2,\frac\nu2\right).
    \]
    Taking $\mu = \beta(n-1)$, $\nu = 1$, and $x = 2/\mu$ gives $\mu x = 2$ and $\mu x + \nu = 3$, so
    \[
    p_n(\beta) = \P(|A_{11}| \ge |A_{21}|) = \P\left(F_{\beta(n-1),1}\le \frac2{\beta(n-1)}\right) = I_{2/3}\left(\frac{\beta(n-1)}2,\frac12\right).
    \]
    For the asymptotics, we use the standard large-parameter behavior of the incomplete beta function: for fixed $z \in (0,1)$ and $b > 0$, as $a \to \infty$,
    \[
    I_z(a,b) \sim \frac{z^a(1-z)^{b-1}}{a\,B(a,b)}, \qquad B(a,b) \sim \Gamma(b)\,a^{-b}.
    \]
    With $z = 2/3$, $b = 1/2$, and $a = \beta(n-1)/2$, we have $(1-z)^{-1/2} = \sqrt3$ and $\Gamma(1/2) = \sqrt\pi$, so
    \[
    p_n(\beta) \sim \frac{\sqrt3}{\sqrt\pi}\sqrt{\frac{2}{\beta(n-1)}}\left(\frac23\right)^{\beta(n-1)/2} = \sqrt{\frac6{\pi\beta(n-1)}}\left(\frac23\right)^{\beta(n-1)/2}. \qedhere
    \]
\end{proof}

For $n=2$ the incomplete beta function is elementary, and the $2\times 2$ real-valued tridiagonal $\beta$-Hermite ensembles have no first pivot needed (i.e., $i_1=1$) with probability
\[
p_2(1)=\frac2\pi\arctan\sqrt2\approx0.6082,\qquad
p_2(2)=1-\frac1{\sqrt3}\approx0.4226,\qquad
p_2(4)=1-\frac4{3\sqrt3}\approx0.2302,
\]
for $\beta=1,2,4$, respectively. \Cref{fig:beta pivots} plots $p_2(\beta)$ for different $\beta$. Note that $p_2(1)$ agrees exactly with the GOE computation $p_2$ from \Cref{sec:background}.

The exact computation above is for the first pivot only, and it already shows how different the tridiagonal ensembles are from the dense ones. Here $i_1 \in \{1,2\}$, so $p_n(\beta) \to 0$ exponentially fast means the first step swaps rows with probability tending to $1$, yet each swap only moves a row by one position. This is the general mechanism for banded matrices, where the limiting permuton is determined by a sublinear bandwidth argument that does not depend on the distribution of the entries.

\begin{theorem}\label{t: banded permuton}
Let $d = d_n \ge 1$ be an integer sequence with $d=o(n)$, and let $A =A_n \in \mathbb R^{n \times n}$ be any random matrix with $A_{ij}=0$ whenever $|i-j|\ge d$. Let $\pi_n=\pi(A)$ be the GEPP permutation. Then
\(        \sum_{j=1}^n |\pi_n(j)-j|
        \le 2(d-1)(n-1)
   \)
surely, and consequently $\mu_{\pi_n}$ converges weakly to the diagonal permuton, the law of $(U,U)$ for $U\sim\Unif([0,1])$.
\end{theorem}

\begin{proof}
    We first show that the pivot row $i_k$ chosen at step $k$ satisfies $i_k \in \{k, k+1, \dots, k+d-1\}$. Row swaps and elimination only mix rows $k,\dots,k+d-1$ at step $k$, so the entries of column $k+1$ below row $k+d$ are untouched and remain zero. Inductively, at step $k$ the entries of column $k$ in rows $i > k+d-1$ are zero. If the active column has a positive maximum, the pivot then lies in the window above; if the active column is zero, our convention gives $i_k=k$. Hence the claimed bound on $i_k$ holds in either case.

    Now $\pi_n$ is a product of the transpositions $(k\ i_k)$, $1 \le k \le n-1$. The footrule $D(\sigma) = \sum_j |\sigma(j)-j|$ satisfies $D(\sigma\tau) \le D(\sigma) + D(\tau)$, and $D((k\ i_k)) = 2|i_k - k| \le 2(d-1)$. Hence $D(\pi_n) \le 2(d-1)(n-1)$. Last, we couple $(J/n, \pi_n(J)/n)$ with $(J/n, J/n)$ for $J \sim \Unif\{1,\dots,n\}$. The expected $\ell^1$ distance is
    \[
    \frac1{n^2}\sum_{j=1}^n |\pi_n(j)-j| \le \frac{2(d-1)}{n} \to 0,
    \]
    and $(J/n,J/n)$ converges weakly to $(U,U)$.
\end{proof}

For $d=2$ this covers the tridiagonal $\beta$-Hermite ensembles for every fixed $\beta>0$, where $i_k \in \{k,k+1\}$ and $\sum_j |\pi_n(j)-j| \le 2(n-1)$. So their GEPP permutations converge to the diagonal permuton, in contrast to the conjectured Lebesgue limit on $[0,1]^2$ for the dense Gaussian ensembles. Future work can also consider limits for GEPP induced permutations for banded random matrices with bandedness parameter $d$ growing proportional to $n$.

\subsection{Discrepancy for GUE pivot comparisons}
\label{sec:discrepancy}

In \cite{P24}, the second author derived the pivot probabilities (i.e., $1 - p_2$) for the GOE and GUE in the $2\times2$ case using the sampling distributions following $A \sim (G + G^*)/\sqrt 2$ but now for $G$ with standard \textit{complex} Gaussian entries for the GUE case, $X \sim N_{\mathbb C}(0,1)$, where $X \stackrel{d}{=}(Y_1 + i Y_2)/\sqrt 2$ for iid $Y_i \sim N(0,1)$. For the GOE, where $A_{11}\sim N(0,2)$ and $A_{21}\sim N(0,1)$ are independent, one obtains again the value $p_2(1)$ (since $\P(|A_{11}| \ge |A_{21}|)
=\P(F_{1,1}\le 2)
=p_2(1)$). For the GUE, the corresponding computation in \cite{P24} gave the same pivot probability as the real-valued $H_2$ construction $1-p_2(2)$. For GUE, since $A_{11}\sim N(0,1)$ and $A_{21}\sim N_{\mathbb C}(0,1)$, this calculation used the standard complex modulus $
|x+iy|=\sqrt{x^2+y^2},
$ corresponding to the $\ell^2$-norm when a complex number is viewed as a vector in $\mathbb R^2$ for the pivot candidate comparison to determine $i_1$. Under this convention, the $H_2$ pivot probability $1-p_2(2)$ is recovered. However, this value does not agree with empirical trials using standard implementations of GEPP in MATLAB, that recover a pivot probability of $2/3$ instead of $1-p_2(2)=1/\sqrt 3 \approx 0.57735$.

The discrepancy arises from the norm used for complex pivot comparisons in standard LAPACK implementations \cite{lapack99}. In particular, the \texttt{ZGETRF} routine used for complex-valued GEPP compares pivot candidates according to $
\|x+iy\|_1=|x|+|y|,
$ rather than the standard complex modulus $|x+i y| = \sqrt{x^2+y^2}$; the corresponding real-valued routine \texttt{DGETRF} uses the usual absolute value. This convention is inherited by implementations such as MATLAB and NumPy/SciPy. (This distinction does not affect the $H_\beta$ ensembles considered above, which are strictly real-valued.) Under the LAPACK convention, the GUE pivot probability is instead exactly $2/3$, in agreement with empirical observations; we derive this value below as a special case of a general formula for the $\ell^1$ and $\ell^2$ pivot probabilities of $\GUE(n)$.

We can now revisit the computations of $p_n = \P(i_1 = 1)$ as seen in \Cref{sec:background} for $A \sim \GUE(n)$, for each of the pivot comparison norm choices.

\begin{proposition}
\label{p: gue pn}
    Let $A \sim \GUE(n)$ and define $p_n^{(k)} = \P(\pi(A)(1) = 1)$ for $\pi(A) = \pi(A,k)$ the GEPP induced permutation that determines pivot candidates using $\ell^k$ evaluations of complex entries. Then $\pi(A) \not\sim \Unif(S_n)$  for all $n \ge 2$ and $k \in \{1,2\}$, while $p_n^{(1)}$ is rational and $p_n^{(2)}$ is irrational, where
    \begin{equation}\label{eq: pn GUE}
        p_n^{(1)} = \frac1{2n-1}, \qquad p_n^{(2)} \sim \frac1{\sqrt{2n \log n}}.
    \end{equation}
\end{proposition}

\begin{proof}
    Since $A \sim \GUE(n)$, we note $A_{11} \sim N(0,1)$ and $A_{j1} \sim N_{\mathbb C}(0,1)$, so $A_{j1}\stackrel{d}{=}(X_{j1} + i X_{j2})/\sqrt 2$ for $X_{ij} \stackrel{iid}{\sim}N(0,1)$. 
    
    For $p_n^{(1)}$, we note first
    \[
    |x_1|+|x_2| = \max(|x_1 + x_2|,|x_1-x_2|) = \sqrt 2 \max(|u|,|v|), \qquad u = \frac{x_1+x_2}{\sqrt 2}, v = \frac{x_1-x_2}{\sqrt 2},
    \]
    while $(u,v)$ is an orthogonal transformation of $(x_1,x_2)$ via a scaled Hadamard transformation. It follows
    \[
    p_n^{(1)} = \P\left(\sqrt 2 |X_{11}| \ge \max_{2 \le j \le n}[|X_{j1}| + |X_{j2}|]\right) = \P\left(|X_{11}| \ge \max_{2 \le j \le n} \max(|U_j|,|V_j|)\right),
    \]
    where we note the $\sqrt 2$ factor can now be canceled from each term. Now $X_{11},U_j,V_j$ for $2 \le j \le n$ comprise $1+2(n-1) = 2n-1$ iid $N(0,1)$ random variables (using the orthogonal invariance of Gaussian vectors). So $p_n^{(1)}$ is just the probability a particular one of these is maximal in magnitude, which is precisely $1/(2n-1)$ by exchangeability. So using the $\ell^1$ norm, then $p_n^{(1)} < 1/n$ for $n \ge 2$ yields $\pi(A) \not\sim \Unif(S_n)$.
    
    Let $Y_j = X_{j1}^2 + X_{j2}^2$ for $2 \le j \le n$; these are iid exponential with mean $2$, so $\P(Y_j \le 2x^2) = 1 - e^{-x^2}$. Conditioning on $X_{11}$ gives
    \[
    p_n^{(2)} = \P\Big(2X_{11}^2 \ge \max_{2 \le j \le n} Y_j\Big) = \mathbb E\big[(1 - e^{-X_{11}^2})^{n-1}\big].
    \]
    Substituting $s = e^{-x^2}$ gives
    \begin{equation}\nonumber
    p_n^{(2)} = \frac1{\sqrt{2\pi}}\int_0^1 s^{-1/2}(-\log s)^{-1/2}(1-s)^{n-1}\,ds \sim \frac1{\sqrt{2n\log n}},
    \end{equation}
    where the asymptotics follow as in \Cref{p: goe pn}: the factor $(1-s)^{n-1}$ localizes the integral to $s \asymp 1/n$, where $(-\log s)^{-1/2} \sim (\log n)^{-1/2}$, leaving 
    $$\frac{1}{\sqrt{2\pi \log n}}\int_0^\infty s^{-1/2}e^{-ns}\,ds = \frac{1}{\sqrt{2n\log n}}.
    $$
    
    To see $p_n^{(2)} \ne 1/n$ for all $n \ge 2$, expand the binomial and use $\mathbb E[e^{-kX^2}] = (1+2k)^{-1/2}$:
    \begin{equation}
    \label{eq: gue exact l2}
    p_n^{(2)} = \sum_{k=0}^{n-1} (-1)^k \binom{n-1}{k} (1+2k)^{-1/2}.
    \end{equation}
    By Bertrand's postulate there is a prime $q = 1+2k_0$ with $n < q < 2n$. No other odd integer $m \le 2n-1$ has squarefree part $q$, since such $m$ would be $j^2 q \ge 9q > 2n-1$. Since square roots of distinct squarefree integers are linearly independent over $\mathbb Q$ (see, e.g., \cite{fractional_powers}), and $q^{-1/2}$ has nonzero coefficient $(-1)^{k_0}\binom{n-1}{k_0}$, the sum is irrational. Hence $p_n^{(2)} \ne 1/n$ and so $\pi(A) \not\sim \Unif(S_n)$ for the $\ell^2$ norm.  
\end{proof}

Moreover, \Cref{p: gue pn} shows both $p_n^{(1)}$ and $p_n^{(2)}$ are each $o_n(1)$, while $p_n^{(1)} = \Theta(1/n)$ and $p_n^{(2)} = \omega(1/n)$. However, unlike as was seen in the GOE case (cf. \Cref{p: goe pn}), we do not have $p_n^{(2)} > 1/n$ until $n = 6$: using \Cref{eq: gue exact l2}, then
\begin{align*}
    p_2^{(2)} &= 1-\frac1{\sqrt 3} \approx 0.4226 < \frac12\\
    p_3^{(2)} &= 1 - \frac2{\sqrt 3}+\frac1{\sqrt 5} \approx 0.2925 < \frac13\\
    p_4^{(2)} &= 1 - \sqrt 3 + \frac3{\sqrt 5} - \frac1{\sqrt 7} \approx 0.2316 < \frac14\\
    p_5^{(2)} &= 1  - \tfrac4{\sqrt3} + \tfrac6{\sqrt5} - \tfrac4{\sqrt7} + \tfrac13 \approx 0.1954 < \frac15\\
    p_6^{(2)} &= 1 - \tfrac5{\sqrt3} + \tfrac{10}{\sqrt5} - \tfrac{10}{\sqrt7} + \tfrac53 - \tfrac1{\sqrt{11}} \approx 0.1709 > \frac16
\end{align*}
We can further show $p_n^{(2)} > 1/n$ for all $n \ge 6$. This is accomplished using direct analysis of the auxiliary function $f(x) = p_x^{(2)} - 1/x$ for real $x > 1/2$, where $p_x^{(2)}$ denotes the integral form above. Write $1/x = \int_0^1 (1-s)^{x-1}\,ds$, so we have
\[
f(x) = \int_0^1 \big(g(s)-1\big)(1-s)^{x-1}\,ds, \qquad g(s) = \frac{1}{\sqrt{2\pi}}\,s^{-1/2}(-\log s)^{-1/2},
\]
and the substitution $s = 1-e^{-u}$ turns this into the Laplace transform $f(x) = \int_0^\infty \big(g(1-e^{-u})-1\big)e^{-xu}\,du$. Now $g(s) > 1$ if and only if $-s\log s < 1/(2\pi)$. The function $-s\log s$ is concave on $[0,1]$, vanishes at both endpoints, and has maximum $1/e > 1/(2\pi)$, so $g-1$ changes sign exactly twice (from $+$ to $-$ to $+$). The generalized Descartes' rule of signs for Laplace transforms then implies that \(f\) has at most two zeros on \((1/2,\infty)\), counted with multiplicity (see, e.g., \cite[Theorem 4]{BhaskarSong2014}). One zero is $x=1$, since $p_1^{(2)} = 1$. Moreover, we have $p_5^{(2)} < 1/5$ and $p_6^{(2)} > 1/6$ by above, so $f$ has a second zero $x_0 \in (5,6)$. These are all the zeros of $f$, so $f(n) > 0$ and hence $p_n^{(2)} > 1/n$ for every integer $n \ge 6$.

Taking $n=2$ in \Cref{eq: pn GUE} recovers exactly the empirically observed discrepancy noted above:

\begin{corollary}\label{prop: 2/3}
    Let $A \sim \GUE(2)$ and GEPP pivot comparisons use the $\ell^1$ norm (as in LAPACK). Then the pivot probability is
    \begin{equation}\nonumber
    \P(\|A_{21}\|_1 > |A_{11}|) = \P(|Z_1| + |Z_2| > \sqrt 2 |Z_3|) = \frac23
\end{equation}
for $Z_i \sim N(0,1)$ iid.
\end{corollary}

\begin{proof}
    By \Cref{p: gue pn}, $p_2^{(1)} = 1/(2\cdot2-1) = 1/3$; since $p_2^{(1)} = \P(\|A_{21}\|_1 \le |A_{11}|)$, an $\ell^1$-pivot movement is needed with probability $1 - 1/3 = 2/3$.
\end{proof}

\begin{remark}[Geometric view of pivoting]
    \Cref{prop: 2/3} also admits a direct geometric derivation (\Cref{fig:pivot-regions}). Writing a standard iid Gaussian vector $\V x \in \mathbb R^3$ in polar form as $\V x = R \V u$ for $\V u \sim \operatorname{Unif}(S^2)$ independent of $R=\|\V x\|_2\sim \chi_3$, the pivot probability reduces to the surface area, relative to the full sphere, of the region $|x|+|y| > \sqrt2|z|$ on $S^2$. \Cref{fig:pivot-regions} compares this region to the $\ell^2$ analogue $\sqrt{x^2+y^2} > \sqrt2|z|$ that the standard complex modulus would give  (with area $1/\sqrt 3 \approx 0.57735$). By symmetry this reduces to a single first-octant integral, that becomes a  trigonometric double integral after switching to spherical coordinates that recovers the $2/3$ value.

    \begin{figure}[t]
        \centering
        \begin{subfigure}[t]{0.35\linewidth}
            \centering
            \includegraphics[height=5cm]{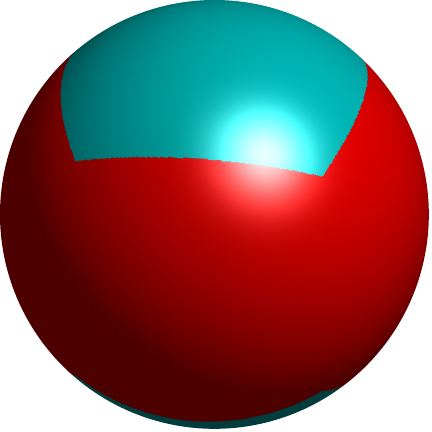}
            \caption{$|x| + |y| > \sqrt{2} |z|$}
        \end{subfigure}
        \qquad
        \begin{subfigure}[t]{0.35\linewidth}
            \centering
            \includegraphics[height=5cm]{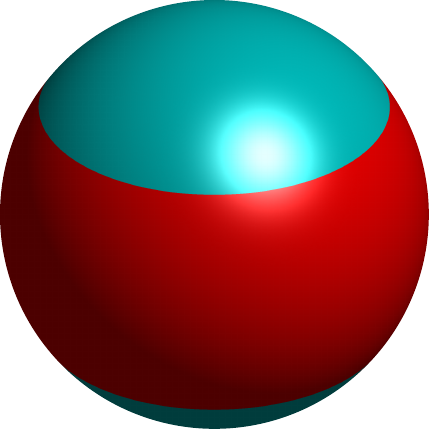}
            \caption{$\sqrt{x^2 + y^2} > \sqrt{2} |z|$}
        \end{subfigure}
        \caption{Regions on the unit sphere defined by different norms used in pivoting comparisons for $\GUE(2)$.}
        \label{fig:pivot-regions}
    \end{figure}
\end{remark}

\section{Discussion and future directions}
\label{sec:discussion}


Conjecture~\ref{conj:goe-permuton} raises a parallel question at the level of the full permutation rather than a single pivot: our supporting evidence for the conjectured uniform permuton limit for Wigner and iid matrices is currently limited to GOE, GUE, and Bernoulli matrices. Extending it to GSE and other $\beta$-ensembles not covered by \Cref{thm:unif_perm} remains open. Such a proof would need to control the global permutation structure directly, rather than the single-pivot statistic $i_1$ used above. A similar question can be extended to random banded matrix ensembles.

The exact butterfly matrix factorization of \Cref{sec:factorization} also leaves open questions once one moves past the cases where closed forms are available. Conjecture~\ref{conj:growth} asks whether the entrywise GEPP growth factor for general butterfly matrices is always attained in the final factor $U$, as it is for simple butterfly matrices \cite{PT23}; the trailing-block cross terms identified in \Cref{sec:growth} appear to be the main obstacle to extending that argument. A related question is how the exact factorization behaves under floating-point arithmetic. Huang and Tikhomirov \cite[Theorem A]{HT23} showed for iid Gaussian matrices that GEPP has polynomial rather than worst-case exponential growth with high probability (excluding underflow and overflow), and that the computed and exact permutation factors agree with very high probability; the corresponding analysis for random butterfly matrices remains open. Similarly, future work will outline limits of butterfly permutations (including general $p$-ary constructions); the self-similar fractal permuton structure aligns with multiplicative cascade constructions rather than the limiting skew Brownian permutons for uniformly random separable permutations (compare \Cref{fig:unif permuton,fig:sparisity plot}). 

The norm dependence uncovered in \Cref{sec:discrepancy} points to a further question about pivoting strategy itself. Pivot movements are a known communication bottleneck for GEPP on parallel architectures \cite{Pa95}, motivating communication-avoiding pivoting strategies such as CALU \cite{calu} and preprocessing schemes that avoid pivoting altogether \cite{baboulin,Pa95,PT23}; the choice of norm for complex-valued pivot comparisons is a further parameter in that design space. While norm-dependent stability effects have been studied in related pivoting settings \cite{Higham98}, their broader implications for stability, and their average-case behavior for random matrix ensembles, remain interesting questions.

The above analysis is limited to the permutations induced using GEPP. Similar questions can be asked of other pivoting strategies like rook or complete pivoting. One could also explore algorithms that permute matrices for other reasons, such as column-pivoted QR or pivoted Cholesky decompositions.

\bibliographystyle{plain}
\bibliography{references}

\end{document}